\documentclass{thesis}
\usepackage[utf8]{inputenc}
\usepackage[english]{babel}
\usepackage{graphicx}
\usepackage[absolute]{textpos} 
\usepackage{amsmath}
\usepackage{amsfonts}
\usepackage{amssymb}
\usepackage{amsthm}
\usepackage{mathtools}	
\usepackage{tikz-cd}	
\usepackage{hyperref}	
\usepackage{enumitem}

\DeclareMathOperator{\id}{id}
\DeclareMathOperator{\im}{im}

\DeclareMathOperator{\Hom}{Hom}

\DeclareMathOperator{\Ext}{Ext}

\newcommand{\F}{\ensuremath{\mathbb{F}}}
\newcommand{\N}{\ensuremath{\mathbb{N}}}
\newcommand{\Z}{\ensuremath{\mathbb{Z}}}

\newcommand{\R}{\ensuremath{\mathbb{R}}}
\newcommand{\C}{\ensuremath{\mathbb{C}}}

\newtheorem{theorem}{Theorem}[section]

\newtheorem{proposition}[theorem]{Proposition}

\newtheorem{lemma}[theorem]{Lemma}
\newtheorem{fact}[theorem]{Fact}

\newtheorem{question}[theorem]{Question}
\theoremstyle{definition}
\newtheorem{definition}[theorem]{Definition}

\theoremstyle{remark}
\newtheorem*{remark}{Remark}

\usepackage{xcolor}

\DeclareMathOperator{\Th}{Th}
\DeclareMathOperator{\Gr}{Gr}

\begin{document}
	\begin{titlepage}
	\begin{textblock}{5}(2.3,1.5)
		\includegraphics[width=0.9\textwidth]{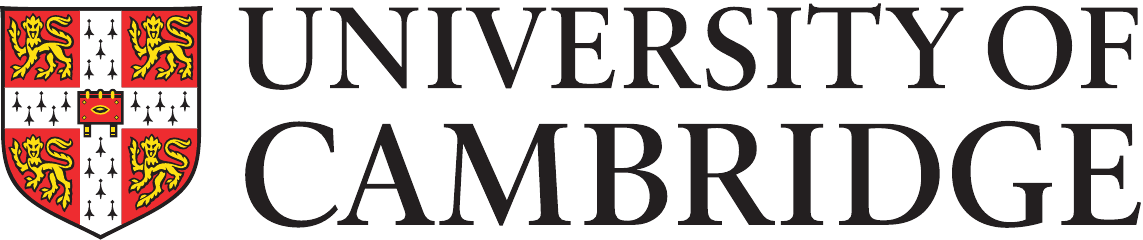}
	\end{textblock}
	
	\begin{center}
		\vspace*{3.5cm}
		
		\textsc{\huge Nilpotence Theorem in Stable}
		
		\vspace{0.3 cm}
		
		\textsc{\huge Homotopy Theory}

		\vspace{1cm}
		
		\textbf{David Popović}
		
		\vspace{0.1cm}
		
		Part III Essay
		
		\vfill
		
		Department of Pure Mathematics and Mathematical Statistics\\
		
		\vspace{0.1 cm}
		
		Supervised by Prof. Oscar Randal-Williams
		
		\vfill
		
	\end{center}
	Cambridge, 2021
\end{titlepage}
	\tableofcontents
	\newpage

\chapter{Introduction}
	Consider the following generic situation in algebraic topology:
	\begin{question}
		Let $X$ and $Y$ be topological spaces and $f: X \to Y$ a map. What can be said about $f$?
	\end{question}
	The answer to the question posed in such generality is `nothing', but plenty of genuinely interesting and tractable questions emerge by restricting $X$ or $Y$ to some class of spaces and imposing conditions on $f$. Some examples include
	\begin{enumerate}[itemsep=0pt, label=$-$]
		\item For which pairs $(n, k)$ does there exist a retraction $\R P^n \to \R P^k$?
		\item What is the order of a given $f:S^k  \to Y$ in $\pi_k(Y)$?
		\item For which $g$ is there a map $ \Sigma_g \to \Sigma_1 $ of positive degree?
	\end{enumerate}
	and perhaps the most fundamental result in this direction would be to establish whether a given map $f: X \to Y$ is non-trivial at all. One's immediate instinct may be to take a generalized homology theory $E_*$ and consider the induced map $E_*(f): E_*(X) \to E_*(Y)$ on homology. If $E_*(f)$ is a non-zero homomorphism, then $f$ is not null-homotopic and our question has been answered.
	
	However, the converse is false -- $E_*(f)$ being the zero map does not imply that $f$ is null-homotopic. There may well exist other homology theories testifying to the contrary. But \emph{a priori} there is no guarantee that we can detect the fact that $f$ is not null-homotopic with \emph{any} theory.
	\begin{question}
		Is there a generalised homology theory $E_*$ such that $f: X \to Y$ is null-homotopic iff $\widetilde{E}_*(f)=0$?
	\end{question}
	This question is still too general to be tractable. Instead suppose that the maps we would like to study are \emph{self-maps} $f: \Sigma^d X \to X$ for some $d \geq 0$ where $\Sigma X$ denotes the reduced suspension of $X$. Such maps can be composed into
	$$f^k: \Sigma ^{kd} \xrightarrow{\Sigma^{(k-1)d}f} \Sigma^{(k-1)d} \to \cdots \to \Sigma X \xrightarrow{f} X.$$
	We call $f$ \emph{nilpotent} if $f^k$ is null-homotopic for some $k$ and we would like to know whether there exists a homology theory detecting that.
	
	\vspace{1em}
	
	Note that two important simplifications were made. We have removed the space $Y$ from consideration and we have transferred the question to the realm of stable homotopy theory where experience suggests that problems become easier. Under these hypotheses we obtain the following remarkable result:
	
	\begin{theorem}[Classical formulation of the nilpotence theorem]
		There is a generalised homology theory $MU_*$ such that a self-map $f: \Sigma^d X \to X$ of a finite CW-complex $X$ is nilpotent iff some iterate of $\widetilde{MU}_*(f)$ is trivial.
	\end{theorem}
	The statement of the theorem says that a certain homology theory -- a coarse algebraic invariant that is supposed to be easily computable -- contains a lot of topological information. Namely, it can determine whether some suspension of $f$ is null-homotopic.
	
	\vspace{1em}
	
	The nilpotence theorem was born as Ravenel's nilpotence conjecture in the early 1980s \cite{hopkins2008mathematical}. Together with Ravenel's other conjectures appearing in his seminal paper \cite{ravenel1984localization}, it guided the direction of research in algebraic topology in the 1980s. The nilpotence conjecture was proven in 1988 by Ethan S. Devinatz, Michael J. Hopkins and Jefferey H. Smith \cite{devinatz1988nilpotence}. In fact, they proved a few closely related and slightly stronger results, one of which is known as the ring spectrum form of the nilpotence theorem.
	
	\begin{theorem}[Nilpotence theorem, ring spectrum form]
		Let $R$ be a ring spectrum and let
		$$h: \pi_*(R) \to MU_*(R)$$
		be the Hurewicz homomorphism. Then every element of $\ker h$ is nilpotent.
	\end{theorem}	
	This essay presents a proof of the theorem. The exposition largely follows the original proof but we deviate from it occasionally. In these instances, our proofs are based on Ravenel's account of the nilpotence theorem in his orange book \cite{ravenel1992nilpotence}.
	
	We do not strive to be concise -- instead, we aim to provide a thorough and well-motivated account of the proof. Whenever using more words allows us to elucidate an argument, we try to do so. This applies both to giving the intuition for the large-scale structure of the proof as well as to supplying technical details whenever they aid understanding.
	
	\vspace{1em}
	
	Having said that, it would be impossible to develop the entire machinery of modern algebraic topology from scratch. We assume the reader is familiar with the Serre and Adams spectral sequences and basic properties of the stable homotopy category, but we recall their properties nonetheless.
	
	\vspace{1em}
	
	The essay is organised as follows: Chapter 2 outlines the background material and sets the stage. The majority of Chapter 3 is dedicated to proving the ring spectrum form of the nilpotence theorem -- the smash product and self-map forms are deduced swiftly at the end. In Chapter 4 we give some applications of the nilpotence theorem and discuss related questions.

\chapter{Background}
	In this chapter we introduce the relevant background material that will be needed for the discussion of the nilpotence theorem and its consequences in the rest of this essay. We begin by introducing the notation and conventions adopted in this work. In Section $2.2$ we state some basic properties of the stable homotopy category $\bf hSp$ and recall definitions of key terms. In Sections $2.3$ and $2.4$ we focus on the James construction and Snaith's splitting. Then we define Thom spectra and introduce $X(n)$ and $F_k$, both of which play crucial roles in the proof. The chapter is concluded with a section about the Adams spectral sequence.  
	
	\section{Notation and conventions}
	This work presents a complicated piece of mathematics with a lot of notation to be defined, used and abused. This section describes the general notation and conventions adopted in this essay.
	
	\vspace{1em}
	
	We mostly work in the topological category $\bf Top$ consisting of weak Hausdorff compactly generated topological spaces and continuous functions. In particular, a \emph{space} means a weak Hausdorff compactly generated topological space and a \emph{map} means a continuous function. Maps labelled by $\hookrightarrow$ are injective and maps labelled by $\twoheadrightarrow$ are surjective. By $\id_X$ we denote the identity map on $X$ and we sometimes omit the subscript if $X$ can be deduced from the context.
	\begin{enumerate}[itemsep=0pt, label=$-$]
		\item $\N = \{  1, 2, 3, \dots  \}$ and $\N_0 = \{ 0, 1, 2, \dots  \}$. A prime number $p$ is fixed.
		\item \emph{Iff} means \emph{if and only if}.
		\item In $\bf Top$, the symbol $\simeq$ denotes a homotopy equivalence and $\cong$ denotes a homeomorphism.
		
		The same symbol $\cong$ is used to denote an isomorphism in any algebraic category. Isomorphisms in the stable homotopy category $\bf hSp$ are called equivalences and are denoted by $\simeq$.
	\end{enumerate}
	Every notational simplification is inherently accompanied by a decrease in the level of precision. Nonetheless, we have adopted the following conventions to aid legibility.
	\begin{enumerate}[itemsep=0pt, label=$-$]
		\item Our spaces are often based, but we are rarely explicit about their basepoints. In particular, $\Sigma X$ denotes the reduced suspension of a based space and $\pi_*(X)$ is $\pi_*(X, x_0)$ for an implicit basepoint $x_0 \in X$.
		\item In $\bf Top$ and $\bf hSp$ we use the same symbol for a map and any of its restrictions and maps on cofibres induced by those. In $\bf hSp$ we additionally use the same symbol for any of (de)suspensions of the original map.
		\item We sometimes denote the spectrum and its $p$-localisation with the same symbol.
	\end{enumerate}
		

	

	\subsection{A note on the exposition}
	In theory, a mathematical proof is a roughly linear sequence of implications. Starting with a state $S_0$, one aims to gradually transform the given hypotheses into conclusions via a sequence of steps
	$$ S_0 \xrightarrow{s_1} S_1 \xrightarrow{s_2} S_2 \xrightarrow{s_3} \dots \xrightarrow{s_n} S_n.$$
	The integer $n$ tends to correlate with the depth and complexity of the argument. When $n$ is large and a mathematician is reading a proof for the first time, they might be unable to understand why $s_1$ brings us any closer to the goal. Their mental representation of the proof could look like this:
	$$ S_1 \xleftarrow{s_1} S_0 \hspace{3cm} S_n.$$
	With that in mind, many authors of mathematical texts write the proofs as follows. Assume $S_{n-1}$ is true. Then we can do $s_n$ and finish. So we only have to show that $S_{n-1}$ is true. Now assume that $S_{n-2}$ is true instead. After applying the reasoning $s_{n-1}$ we can reach the state $S_{n-1}$. But we have previously shown that we are done once we reach $S_{n-1}$. Hence the goal now is to show that $S_{n-2}$ is true. For this assume $S_{n-3}$...
	
	\vspace{1em}	
	
	One can quickly get lost in the jungle of things that are true, things one wishes to be true and things assumed to be true at any given stage of the proof. For this reason, it is my strong personal preference \emph{not} to do that.
	
	Instead, we have opted to explain the global structure of the proof before delving into the details. Introducing significant intermediate goals $S_a$, $S_b$ and $S_c$ prior to the beginning of the proof allows one to appreciate why $s_1$ brings us closer to $S_a$ even if the path to $S_n$ is still hazy.
	
	\vspace{1em}
	
	While the exposition of the proof might be somewhat original, the content is not. We mostly follow the original article by Devinatz-Hopkins-Smith \cite{devinatz1988nilpotence}, which is, to the best of our knowledge, the only complete account of the proof in the literature. The other invaluable reference is Ravenel's sketch in his orange book \cite{ravenel1992nilpotence}. Most of this proof closely follows \cite{devinatz1988nilpotence} with significantly fewer details, but differs from the original in a proof of an important algebraic lemma. At this stage we adopt Ravenel's geometric approach, because it successfully circumvents one of the very technical parts of the original proof.
	
	\vspace{1em}
	
	This essay has also been greatly influenced by Cary Malkiewich's excellent introduction to the stable homotopy category \cite{malkiewich2014stable}, Matthew Akhil's insightful post on MathOverflow \cite{mathew}, Eric Petersen's informal blog entries about the nilpotence theorem \cite{peterson} and the discussions I have had with Oscar.
	
	\section{Stable homotopy category}  
	We mostly work with axiomatic properties of the stable homotopy category $\bf hSp$ and not in some model category of spectra. Whenever a model is needed, we resort to the sequential spectra. In this section we recall the basic properties of $\bf hSp$ following \cite{malkiewich2014stable} and introduce some terminology.
	
	\vspace{1em}
	
	Consider the Quillen model structure on $\bf Top$ in which the weak equivalences are the maps inducing isomorphisms on homotopy groups and fibrations are the Serre fibrations. The class of fibrant objects of $\bf Top$ contains all spaces and the cofibrant objects are the retracts of the CW complexes. Therefore we may form the homotopy category $\bf hTop$ in which the objects are the CW complexes and the morphisms are the homotopy classes of morphisms of $\bf Top$.
	
	Analogously one defines the homotopy category ${\bf hTop}_*$ where ${\bf Top}_*$ is the category of based topological spaces and basepoint-preserving maps.
	
	\vspace{1em}
	
	We now list a few properties of the stable homotopy category $\bf hSp$ for completeness and future use in this essay. The proofs of these properties are obtained by choosing a model category of spectra (for example sequential spectra, the category described by Adams \cite{adams1974stable}, othogonal spectra, symmetric spectra \emph{etc.}) and verifying from there.  
	\begin{fact}
		There is a stabilization functor $\Sigma^\infty: {\bf hTop}_* \to {\bf hSp}$ and it has a right adjoint $\Omega^\infty: {\bf hSp} \to {\bf hTop}_*$.
	\end{fact}
	\begin{remark}
		It is a common practice in homotopy theory to denote both a based space $X$ and its suspension spectrum $\Sigma^{\infty} X$ by $X$. The context usually prevents ambiguities, but we prefer to be very explicit about the object we have in mind. This is why we do not adopt this convention, except for the sphere spectrum $S=\Sigma^\infty S^0$ and its suspensions $S^n = \Sigma^\infty S^n$.
	\end{remark}
	\begin{fact}
		There are suspension and loopspace functors $\Sigma, \Omega : \bf hSp \to hSp$ which are inverse equivalences. They agree with the usual reduced suspension and based loopspace functors in ${\bf hTop}_*$ in the sense that the diagrams
		$$
		\begin{tikzcd}
		{\bf hTop}_* \arrow[r, "\Sigma"] \arrow[d, "\Sigma^\infty"] &  {\bf hTop}_* \arrow[d, "\Sigma^\infty"]\\
		{\bf hSp} \arrow[r, "\Sigma"]& {\bf hSp}
		\end{tikzcd}
		\quad \text{and} \quad
		\begin{tikzcd}
		{\bf hTop}_*   & {\bf hTop}_* \arrow[l, "\Omega"]  \\
		{\bf hSp} \arrow[u, "\Omega^\infty"] & {\bf hSp} \arrow[u, "\Omega ^\infty"] \arrow[l, "\Omega"]
		\end{tikzcd}
		$$
		commute.
	\end{fact}
	\begin{fact}
		For any $X, Y \in \bf hSp$, the set of morphisms $\left[X, Y\right] := \Hom_{\bf hSp}(X, Y)$ has the structure of an abelian group. The category $\bf hSp$ contains finite products $X \times Y$, coproducts $X \vee Y$ and the zero object $*$. There are natural isomorphisms
		\begin{align*}
			X \vee * &\to X\\
			X &\to X \times *\\
			X \vee Y &\to X \times Y
		\end{align*}
		induced by the unique maps $* \to X \to *$ making $\bf hSp$ into an additive category.
		
		Moreover, $\bf hSp$ is equipped with a tensor product given by the smash product of spectra $X \land Y$ whose unit is $S$. There are natural isomorphisms $X \land Y \cong Y \land X$, $X \land (Y \land Z) \cong (X \land Y) \land Z$ and the smash product additionally satisfies the triangle, pentagon and hexagon identities. This makes $\bf hSp$ into a symmetric monoidal category. 
	\end{fact}
	\begin{fact}
		In $\bf hSp$ a sequence of morphisms is a homotopy fibre sequence iff it is a homotopy cofibre sequence.
	\end{fact}
	\begin{definition}
		Let $E \in \bf hSp$ be a spectrum. 
		\begin{enumerate}[itemsep=0pt, label=$-$]
			\item A spectrum $X$ is $E$-acyclic if $E \land X \simeq 0$.
			\item A morphism $f: X \to Y$ of spectra is an $E$-equivalence if $$\id_E \land f: E \land X \to E \land Y$$ is an equivalence.
		\end{enumerate}
	\end{definition}
	Spectra represent generalized homology and cohomology theories via Brown's representability theorem. From this perspective $f$ is an $E$-equivalence if the induced map on homology
	$$E_*(f): E_*(X) \to E_*(Y)$$
	is an isomorphism.
	\begin{definition}
		Let $E \in \bf hSp$ be a spectrum. A spectrum $X$ is a $E$-local if for every $E$-equivalence $f: Y \to Z$ the map $\left[f, X\right]_*:\left[Z, X\right]_* \to \left[Y, X\right]_*$ is an isomorphism for all $*$. 
	\end{definition}	
	\begin{definition}
		Let $E, X \in \bf hSp$. Then an $E$-localization of $X$ is an $E$-equivalence $X \to L_EX$ where $L_EX$ is some $E$-local spectrum.
	\end{definition}
	\begin{fact}
		For any $E, X \in \bf hSp$ an $E$-localisation of $X$ exists.
	\end{fact}
	The process of passing from a spectrum to an $E$-local spectrum is called the Bousfield localization of spectra. Categorically this is a localization of $\bf hSp$ at the collection $E$-equivalences.
		
	\vspace{1em}
	
	In this essay we localize only at the Moore spectrum $S\Z_{(p)}$ of $\Z_{(p)}$. In this case we also write $X_{(p)} := L_{S \Z_{(p)}}X$ and refer to it as the $p$-localization of $X$. For any spectrum $E$ we now have $E_*(X_{(p)}) = E_*(X) \otimes \Z_{(p)}$ and it follows that $X$ is contractible iff $X_{(p)}$ is contractible for every prime $p$.
	\begin{definition}
		A spectrum $X$ is contractible if $\pi_*(X) = 0$.
	\end{definition}
	\begin{definition} Let $N \in \Z$. A spectrum $X$ is
		\begin{enumerate}[itemsep=0pt, label=$-$]
			\item $N$-connected if $\pi_d(X)=0$ for all $d \leq N$,
			\item connective if it is $N$-connected for some $N$.
		\end{enumerate}
	\end{definition}
	\begin{remark}
		Some authors define connective to mean $(-1)$-connected and there is no widespread agreement about which definition to use.
	\end{remark}
	\begin{definition}
		A ring spectrum is a ring object  $(R, \eta, m)$ in $\bf hSp$.
	\end{definition}
	Here $\eta: S \to R$ is a unit and $m: R \land R \to R$ is a multiplication map. The triple $(R, \eta, m)$ will usually be shortened to just $R$.
	\begin{definition}
		Let $R$ be a ring spectrum and $\alpha \in \pi_d(R)$. There is an induced map
		$$ \Sigma^d R \simeq S^d \land R \xrightarrow{\alpha \land \id} R \land R \xrightarrow{m} R  $$
		which we also denote by $\alpha$. Then the telescope $\alpha^{-1} R$ is the homotopy colimit of
		$$ R \xrightarrow{\alpha} \Sigma^{-d}R \xrightarrow{\alpha} \Sigma^{-2d} R \to \cdots.$$
	\end{definition}
	\begin{remark}
		 By a homotopy colimit we mean the following: the diagram above can be lifted to a sequence of cofibrations between cofibrant objects in some model category of spectra. The image in $\bf hSp$ of the categorical colimit of this lift is independent of the lift and called the homotopy colimit of the diagram. Despite the name, the homotopy colimit is \emph{not} the categorical colimit in $\bf hSp$.
	\end{remark}
	\begin{remark}
		Smash product commutes with homotopy colimits. Taking homotopy groups commutes with filtered homotopy colimits. Therefore both of these constructions commute with taking the telescopes.
	\end{remark}
	
	\begin{definition}
		Spectra $E, F \in \bf hSp$ are Bousfield equivalent if for every spectrum $X$ we have that $E \land X \simeq *$ iff $F \land X \simeq *$.
		
		\vspace{0.2 cm}
		
		\noindent The Bousfield equivalence class of $E$ is denoted by $\langle E \rangle$. We write $\langle E \rangle \geq \langle F \rangle$ if for each spectrum $X$ we have that $E \land X \simeq *$ implies $F \land X\simeq *$. We further define $\langle E \rangle \land \langle F \rangle = \langle E \land F \rangle$ and $\langle E \rangle \vee \langle F \rangle = \langle E \vee F \rangle$.
	\end{definition}
	
	\begin{definition}
		A spectrum $X \in \bf hSp$ is of finite type if $\pi_d(X)$ is finitely generated for each $d$. It is finite if it is equivalent to $\Sigma^{-N}\Sigma^\infty Y$ for some $N \in \N_0$ and some finite based CW complex $Y$.
	\end{definition}
	
	We now briefly discuss the Spanier-Whitehead duality. The geometric idea is that a space $X$ can be considered as dual to its complement in $S^N$ for a large $N$. This is formalized in the language of spectra with the following theorem.
	\begin{theorem}[Spanier-Whitehead Duality]
		For any finite spectrum $X$, there is a finite spectrum $DX$ such that
		\begin{enumerate}[itemsep=0pt, label=$-$]
			\item for any spectrum $Y$, there is an isomorphism of graded abelian groups $\left[X, Y\right]_* \to \pi_*(DX \land Y)$ that is natural in both $X$ and $Y$.
			
			The maps $S^n \land X \to Y$ and $S^n \to DX \land Y$ corresponding under this isomorphism are called adjoint.
			
			\item $D(X \land Y) = DX \land DY$.
			\item $DDX \simeq X$ and $\left[X, Y\right]_* \cong \left[DY, DX\right]_*$.
			\item $X \mapsto DX$ is a contravariant functor.
		\end{enumerate}
	\end{theorem}
	In this essay we mostly use the duality to replace maps $\Sigma^n X \to Y$ with maps with domain $S^n$ and thus simplify the setting while retaining all essential information.
	
	\section{James construction}
	In this section, we define the James construction and the James-Hopf maps. The James construction lets us understand the homotopy type of the spaces $\Omega \Sigma X$ geometrically and plays an important role in algebraic topology beyond this proof.
	
	\vspace{1em}
	
	
	The James construction $JX$ is a free topological monoid on a based space $(X, *)$. Formally, we have the following definition.
	
	\begin{definition}
		Let $(X, *)$ be a based space. The James construction is the space $JX= \bigsqcup_{j=0}^\infty X^j / \sim$ where $\sim$ is the equivalence relation generated by
		$$ (x_1, \dots, x_{i-1}, *, x_i, \dots, x_j) \sim (x_1, \dots, x_i, x_{i+1}, \dots, x_j)$$
		for each $i$ and $j$.
		
		The $k$-th stage of the James construction on $X$ is the space $J_kX= \bigsqcup_{j=0}^k X^j / \sim$ where $\sim$ is the restriction of the above equivalence relation.
	\end{definition}
	The space $JX$ is a monoid in which multiplication is given by the concatenation of words and whose identity is the basepoint $*$. The $k$-th stage of the James construction $J_kX$ is a subspace containing all words of length at most $k$. The importance of the James construction stems from the following result.
	\begin{theorem}\label{James}
		If $X$ is a connected $CW$ complex, then
		\begin{itemize}[label=$-$]
			\item $JX \simeq \Omega \Sigma X$ and
			\item $\Sigma J_kX \simeq \bigvee_{j=0}^k \Sigma X^{\land j}$ and $\Sigma JX \simeq \bigvee_{j=0}^\infty \Sigma X^{\land j}$.
		\end{itemize}
	\end{theorem}
	Here and elsewhere in the essay, an expression of the form $(\cdot)^{\land j}$ refers to the $j$-fold smash product of spaces, maps or spectra.
	
	\vspace{1em}
	
	The spaces $J_kX$ form a filtration of the space $X$. We have $J_k X / J_{k-1} X \simeq X^{\land k}$ as is easily seen from the definition of the equivalence relation $\sim$.
	
	\vspace{1 em}
	
	In this essay, the Theorem \ref{James} is used primarily in the case $X=S^{2m}$ for studying the space $\Omega S^{2m+1}$. The homology of this loop space can be obtained by a standard calculation with the Serre spectral sequence, but this argument sheds no light on the geometric structure of $\Omega S^{2m+1}$. The James construction equips its homotopy type with the structure of a CW complex with one cell in each dimension divisible by $2m$.
	
	\vspace{1 em}
	
	The James-Hopf maps generalize the Hopf invariant.
	\begin{definition}\label{JamesHopf}
		Let $k \in \N$. Consider the James splitting map composed with the projection
		$$ \Sigma J X \xrightarrow{\simeq} \Sigma \bigvee_{j=0}^\infty X^{\land j} \to \Sigma X^{\land k}.$$
		The functors $\Sigma$ and $\Omega$ are an adjoint pair and we define the adjoint map
		$$ JX \to \Omega \Sigma X^{\land k} $$
		to be the James-Hopf map.
	\end{definition}
	Fixing the coefficients in a field $F$, the Künneth theorem yields an isomorphism $H_*(Y \times Y; F) \cong H_*(Y; F) \otimes H_*(Y; F)$ and thus equips homology $H_*(Y;F)$ of any space $Y$ with the coalgebra structure induced by the diagonal map $Y \to Y \times Y$. Then $H_*(\Omega \Sigma X ; F) \to H_*(\Omega \Sigma X^{\land k};F)$ becomes a map of coalgebras. When $X$ is a sphere, this map can be explicitly calculated in terms of the generators \cite[see Lecture $4$, §3]{akhil2012spectra}. We shall need to know is that if $X=S^{2m}$ and $F = \F_p$, the map is surjective.
	
	\section{Snaith's splitting}
	Snaith's result translates the Theorem \ref{James} to the stable homotopy category.
	\begin{theorem}[Snaith's splitting]\label{Snaith}
		Let $n \in \N$. For any based CW complex $X$ there is a splitting
		$$ \Sigma^\infty \Omega^n \Sigma^n X \simeq \bigvee_{k=0}^\infty D_k$$
		where $D_k$ are some finite spectra.
	\end{theorem}
 	The James construction $JX$ can be thought of as the (unstable!) case $n=1$. There are also concrete models for the homotopy types of spaces $\Omega^n \Sigma^n X$ for $n \geq 2$ using the theory of operads, but this is not discussed further in this essay.	
	
	\section{Thom spectra}
	The nilpotence theorem is a statement about $MU$ detecting nilpotence. The spectrum $MU$ is the spectrum associated to the generalized cohomology theory complex cobordism via Brown's representability theorem. In this section we provide an alternative construction of $MU$, define other Thom spectra featuring in the essay and establish some of their properties.
	
	\vspace{1em}
	
	Let $p: E \to X$ be a complex vector bundle with an inner product $\langle \cdot, \cdot \rangle$. Recall that any vector bundle over a paracompact Hausdorff space admits an inner product and all spaces we consider have these properties. Define the disc bundle
	$$ D_X(E)=\{  e \in E \ | \ \langle e, e \rangle \leq 1  \} $$
	and the sphere bundle
	$$ S_X(E)=\{  e \in E \ | \ \langle e, e \rangle = 1 \}. $$
	The Thom space of $E$ is
	$$ \Th_X(E)=D_X(E)/S_X(E). $$
	One can also consider a slightly more general construction. For a pair of spaces $(X, A)$ define the relative Thom space of $E$ as the cofibre
	$$ \Th_{X/A}(E) = D_X(E)/(S_X(E) \cup D_A(E)).$$
	We can also Thomify maps.
	\begin{definition}
		Let $f: Y \to X$ be a map of spaces. It defines a pullback bundle $f^*E$ over $Y$ and we have the commutative diagram
		$$\begin{tikzcd}
		f^*E \arrow[r] \arrow[d]& E \arrow[d, "p"]\\
		Y \arrow[r, "f"] & X.
		\end{tikzcd}$$
		The induced map $\Th(f): \Th_Y(f^*E) \to \Th_X(E)$ is the Thomification of $f$.
	\end{definition}

	What we are really interested in is the notion of the Thom spectrum - the stable analogue of the Thom space. We shall define the Thom spectrum of a map $f$ as a sequence of Thom spaces and structure maps associated to certain bundles related to $f$. Let us introduce these bundles.
	
	\vspace{1em}
	
	Let $G$ be a topological group. A classifying space $BG$ is a space with the property that for any CW complex $Y$ there is a bijection
	$$ 
	\begin{tikzcd}
	\left[Y, BG\right] \arrow[r, leftrightarrow]& \{ \text{principal $G$-bundles over $Y$} \}
	\end{tikzcd}
	$$
	between the set of homotopy classes of maps $Y \to BG$ and the set of isomorphism classes of principal $G$-bundles over $Y$.
	
	The question of existence of $BG$ can profitably be rephrased as the question whether the functor $\bf hTop \to Set$ given  by $Y \mapsto \{ \text{principal $G$-bundles over $Y$} \}$ is representable. This follows from the Brown's representability theorem. The space $BG$ is unique in $\bf hTop$ by the Yoneda lemma so by CW approximation $BG$ is unique in $\bf Top$ up to a weak homotopy equivalence.
	
	\vspace{1em}
	
	To define Thom spectra consider the classifying space of the unitary group $U(k)$, in which case the abstract machinery can be replaced by an explicit model. We have that $BU(k) = \Gr_k(\C^\infty)$ is the infinite Grassmannian.
	
	\begin{theorem}
		Let $\gamma_k(\C^\infty) \to \Gr_k(\C^\infty)$ be the tautological $k$-dimensional complex vector bundle. For any CW complex $Y$ there is a bijection
		\begin{align*}
			\left[Y, \Gr_k(\C^\infty) \right] & \to \{ \text{$k$-dimensional complex vector bundles over $Y$} \}\\
			f &\mapsto f^*\gamma_k(\C^\infty).
		\end{align*}
		Hence $BU(k)=\Gr_k(\C^\infty)$. 
	\end{theorem}
	
	With the existence of $\gamma_k(\C^\infty)$ we can finally define Thom spectra. We first give a convenient models for the spaces $BU$ and $BU(k)$. Let $BU$ be the infinite mapping telescope of
	$$ \Gr_1(\C^\infty) \hookrightarrow \Gr_2(\C^\infty) \hookrightarrow \cdots$$
	where the inclusions are induced by the maps $\Gr_k(\C^m) \to \Gr_{k+1}(\C^{m+1})$ given by $V \mapsto \textrm{span}\{V, e_{m+1}\}$ for $e_{m+1} \notin \C^m$. Similarly let $BU^{(k)}$ denote the finite mapping telescope of
	$$ \Gr_1(\C^\infty) \hookrightarrow \Gr_2(\C^\infty) \hookrightarrow \dots \hookrightarrow \Gr_k(\C^\infty)$$
	which is homotopy equivalent to $BU(k) = \Gr_k(\C^\infty)$ by collapsing the telescope to its right-hand end. Pulling back $\gamma_k(\C^\infty)$ along this map gives a bundle $V_k$ and by restricting further along the inclusion $BU^{(k-1)} \hookrightarrow BU^{(k)}$ we obtain a commutative cube
	$$
	\begin{tikzcd}[row sep=scriptsize, column sep=scriptsize]
	& V_{k-1} \oplus \C \arrow[dl, "\simeq"] \arrow[rr] \arrow[dd] & & V_k \arrow[dl, "\simeq"] \arrow[dd] \\
	\gamma_{k-1}(\C^\infty) \oplus \C \arrow[rr, crossing over] \arrow[dd] & & \gamma_k(\C^\infty) \\
	& BU^{(k-1)} \arrow[dl, "\simeq"] \arrow[hook, rr] & &  BU^{(k)} \arrow[dl, "\simeq"] \\
	\Gr_{k-1}(\C^\infty) \arrow[rr, hook] & & \Gr_k(\C^\infty). \arrow[from=uu, crossing over]
	\end{tikzcd}
	$$
	To see that the restrictions of $\gamma_k(\C^\infty)$ and of $V_k$ are indeed as in the diagram, these pullbacks can be computed manually. For example, the pullback of $\gamma_k(\C^\infty)$ is given by
	$$\{ (V, x, U) \in \Gr_{k-1}(\C^\infty) \times \gamma_k(\C^\infty) \ | \ \mathrm{span}\{ V, e_{m+1} \} = U \text{ where } U \leq \C^m\}.$$
	Note that $x \in U$ can be uniquely expressed as $x=x'+x''$ where $x' \in V$ and $x'' \in \mathrm{span}\{ e_{m+1} \}$. Thence the bundle is isomorphic to
	$$\{ (V,x',x'') \in \gamma_{k-1}(\C^\infty) \times \mathrm{span}\{ e_{m+1} \} \} \cong \gamma_{k-1}(\C^\infty) \oplus \C$$
	as required.
	
	
	\vspace{1em}
	
	Let $f:Y \to BU$ be a map. For any $k \in \N$ we define the preimages $Y^{(k)} := f^{-1}(BU^{(k)})$. Restricting $f$ to these subspaces yields maps into $BU^{(k)}$ and the pullback square can be extended to
	$$\begin{tikzcd}
	f^*(V_k) \arrow[r] \arrow[d]&V_k \arrow[r] \arrow[d]& \gamma_k(\C^\infty) \arrow[d]\\
	Y^{(k)} \arrow[r, "f"]&BU^{(k)} \arrow[r] & \Gr_k(\C^\infty).
	\end{tikzcd}$$
	Restricting to $\Gr_{k-1}(\C^\infty)$ extends the commutative cube above to the commutative cuboid with $f^*(V_k) \cong f^*(V_{k-1}) \oplus \C$.
	\begin{definition}
		The Thom spectrum of the map $f: Y \to BU$ is denoted by $Y^f$ and has spaces
		\begin{align*}
		Y^f_{2k} &= \Th_{Y^{(k)}}(f^* (V_k))\\
		Y^f_{2k+1} &= \Sigma Y^f_{2k}
		\end{align*}
		and the structure maps $\Sigma Y^f_{2k} \xrightarrow{\id} Y^f_{2k+1}$ and 
		$$\Sigma Y^f_{2k+1} = \Sigma^2 Y^f_{2k} = \Sigma^2 \Th_{Y^{(k)}} (f^*(V_k)) \xrightarrow{\cong} \Th_{Y^{(k+1)}}(f^*(V_{k+1})) = Y^f_{2k+2}.$$	
	\end{definition}
	This construction carries over to the relative version for the pair of spaces. Note that any complex vector bundle is orientable as a real vector bundle. The main tool for calculation of homology groups of Thom spectra is the Thom isomorphism theorem.
	\begin{theorem}[Thom isomorphism theorem]
		Let $p: E \to X$ be a complex vector bundle of complex rank $k$. There is a Thom class $u \in H^{2k} (X; \Z)$ such that taking the cap product with $u$
		$$ \widetilde{H}_{i+2k} (\Th_X(E); \Z) \to H_i(X;\Z)$$
		is an isomorphism.
	\end{theorem}	
	Consider the category ${\bf Top}_{BU}$ of spaces over $BU$ whose objects are maps $Y \xrightarrow{f} BU$ and whose morphisms are commutative diagrams
	$$
	\begin{tikzcd}
	Y \arrow[rr] \arrow[dr, "f"]& &Y' \arrow[dl, "f'"]\\
	&BU&
	\end{tikzcd}.
	$$
	Taking the Thom spectrum is a functor from the category ${\bf Top}_{BU}$ to the category $\bf hSp$. In this language, the Thom isomorphism theorem together with passing to the direct limits shows that the induced map on integral homology in ${\bf Top}_{BU}$ and $\bf hSp$ is the same.
	
	\vspace{1em}
	
	In the rest of this section on Thom spectra we turn away from the general theory and instead focus on increasingly specific objects. We introduce the main characters $MU$, $X(n)$, $F_k$, $G_j$ of the story that unfolds in Chapter 3.  
	\begin{definition}
		The spectrum $MU$ is the Thom spectrum associated to the identity map $BU \xrightarrow{\id} BU$.
	\end{definition}
	
	\begin{definition}
		The spectrum $X(n)$ is the Thom spectrum associated to the composite map $\Omega SU(n) \hookrightarrow \Omega SU \simeq BU$ where the second map is a homotopy equivalence by the Bott periodicity theorem.
	\end{definition}
	\begin{lemma}
		$X(n)$ and $MU$ are commutative ring spectra with $H_*(X(n); \Z) \cong \Z [ x_1,\dots, x_{n-1}]$ and $H_*(MU; \Z) \cong \Z [ x_1, x_2, \dots ]$ where $x_i$ is a generator of degree $2i$.
	\end{lemma}
	\begin{proof}
		For any space $X$, the loop space $\Omega X$ is an $H$-space. This means that there is a map $\mu: \Omega X \times \Omega X \to \Omega X$ given by concatenation of loops. It is associative up to homotopy and the constant loop is its identity. By using the cross product this defines a strictly associative unital map $$ H_*(\Omega X;\Z) \otimes H_*(\Omega X;\Z) \xrightarrow{\times} H_*(\Omega X \times \Omega X;\Z) \xrightarrow{\mu_*} H_*(\Omega X;\Z)$$ on homology. This multiplication equips $H_*(\Omega X;\Z)$ with the ring structure. Whenever $X$ itself is a topological group, this product is commutative. Passing to $\bf hSp$ using the Thom isomorphism theorem we see that $\mu$ makes $X(n)$ and $MU$ into commutative ring spectra with $H_*(X(n); \Z) \cong H_*(\Omega SU(n);\Z)$ and $H_*(MU; \Z) \cong H_*(\Omega SU;\Z)$.
		
		The structure of the homology ring $H_*(\Omega SU(n);\Z) \cong \Z [ x_1,\dots, x_{n-1}]$ is obtained by studying the homological Serre spectral sequence for the path fibration $\Omega SU(n) \to P_*SU(n) \to SU(n)$ in which the homology of $SU(n)$ is well-known (or can be obtained using yet another Serre spectral sequence argument). The corresponding result for $H_*(\Omega SU;\Z)$ follows by passing to the evident direct limit.
	\end{proof}
	The following lemma is now immediate.
	\begin{lemma}\label{XnMU}
		$X(n) \to MU$ is $(2n-1)$-connected for any $n \in \N$.
	\end{lemma}
	\begin{proof}
		Consider the map $\Omega SU(n) \hookrightarrow \Omega SU \simeq BU$ in ${\bf Top}_{BU}$. On integral homology this is the inclusion 
		$$H_*(\Omega SU(n);\Z) \cong \Z\left[x_1, \dots, x_{n-1}\right] \hookrightarrow \Z \left[x_1, x_2, \dots \right] \cong H_*(BU;\Z)$$
		and by the Thom isomorphism theorem there is the same effect on homology in $\bf hSp$ after passing to Thom spectra. Using the quantitative version of the homology Whitehead theorem we conclude that the map is $(2n-1)$-connected as required.
	\end{proof}
	In other words, the sequence of inclusions $\Omega SU(n) \to \Omega SU(n+1)$ in ${\bf Top}_{BU}$ gives rise to a sequence of ring spectra maps $\dots \to X(n) \to X(n+1) \to \cdots$ that form a filtration of $MU$.
	\begin{remark}
		Lemma \ref{XnMU} shows that $MU$ can be thought of as $X(\infty)$. On the other hand, $X(1)=S$ so the spectra $X(n)$ could be interpreted as interpolating steps between $S$ and $MU$. We later expand on this remark dramatically and see that this is precisely the perspective adopted in the proof of the nilpotence theorem. 
	\end{remark}
	We now define a further refinement of the spectra $X(n)$. Fix a unit vector $u \in \C^{n+1}$ and consider the fibration
	\begin{align*}
	SU(n) \to SU(n+1) &\xrightarrow{e} S^{2n+1}\\
	A& \mapsto Au.
	\end{align*}
	Applying the loop space functor $\Omega$ we obtain the fibration
	$$ \Omega SU(n) \to \Omega SU(n+1) \xrightarrow{\Omega e} \Omega S^{2n+1}$$
	and recall that $\Omega S^{2n+1} \simeq JS^{2n}$ where $JS^{2n}$ denotes the James construction on $S^{2n}$. The inclusion of a $2nk$-skeleton $J_kS^{2n} \hookrightarrow JS^{2n}$ defines the pullback bundle $B_k := i^*\Omega SU(n+1)$ and we can draw the diagram
	$$
	\begin{tikzcd}
	\Omega SU(n) \arrow[r, equal, "\id"] \arrow[d] & \Omega SU(n) \arrow[d]\\
	B_k := i^*\Omega SU(n+1) \arrow[d] \arrow[r, "i_{-1,0}"] & \Omega SU(n+1) \arrow[d, "\Omega e"]\\
	J_kS^{2n} \arrow[r, hook, "i"] & JS^{2n} \simeq \Omega S^{2n+1}  
	\end{tikzcd}
	$$
	noting that $B_k$ is only defined up to homotopy equivalence unless a particular homotopy equivalence $JS^{2n} \to \Omega S^{2n+1}$ is chosen.
	
	\vspace{1em}
	
	There is a compelling reason for the unusual name $i_{-1, 0}$ of the canonical map $B_k \to \Omega SU(n+1)$ in the diagram. We later encounter the maps $i_{s,t}$ for more general $s$ and $t$ and the map $i_{-1, 0}$ fits into that framework.

	\begin{definition}
		Let $F_k$ be the Thom spectrum associated to the map $B_k = i^*\Omega SU(n+1) \to \Omega SU(n+1) \to \Omega SU \xrightarrow{\simeq} BU$.
	\end{definition}
	Note that the filtration $J_0S^{2n} \subset J_1S^{2n} \subset \dots $ of $JS^{2n}$ by the partial James constructions induces by taking pullback fibre bundles the sequence of maps
	$$\Omega SU(n) = B_0 \to B_1 \to \dots \to \Omega SU(n+1).$$
	Passing to Thom spectra yields a filtration
	$$ X(n) = F_0 \to F_1 \to \dots \to X(n+1)$$
	 of $X(n+1)$. This hints at the role that the spectra $F_k$ assume in the proof of the nilpotence theorem. They serve as intermediate steps when passing between $X(n+1)$ and $X(n)$. Phrasing this in a formal language we obtain:
	\begin{lemma}\label{homologyofBk}
		The spectra $F_k$ are $X(n)$-module spectra and $H_*(F_k; \Z)$ is a free $H_*(X(n);\Z) \cong \Z [x_1, \dots, x_{n-1}]$-submodule of $\Z [x_1, \dots, x_n]$ generated by $1, x_n, \dots, x_n^{k}$.
	\end{lemma}
	\begin{proof}
		As usual we prove the result in ${\bf Top}_{BU}$ and then pass to $\bf hSp$ with the Thom isomorphism theorem. Any map $f: X \to Y$ in $\bf Top$ can be replaced by a fibration and then $\Omega Y$ acts on the homotopy fibre of $f$ by concatenation of loops. This construction is revisited and explained in greater depth in Section 3.3.3. The fibration $\Omega SU(n) \to B_k \to J_kS^{2n}$ extends to the right to the homotopy fibre sequence $\Omega SU(n) \to B_k \to J_kS^{2n}\to SU(n)$ and so by the above construction $\Omega SU(n)$ acts on the homotopy fibre of $J_kS^{2n} \to SU(n)$ and hence on $B_k$. This action makes $F_k$ into an $X(n)$-module spectrum after passing to Thom spectra and equips $H_*(F_k; \Z)$ with the $H_*(X(n); \Z)$-module structure.
		
		\vspace{1em}	
		
		To calculate the homology $H_*(F_k; \Z)$ consider the Serre spectral sequence for integral cohomology associated to the fibration $\Omega SU(n) \to B_k \to J_kS^{2n}$. The cohomology of both the base and the fibre is known and concentrated in even degrees, so all differentials in the spectral sequence are zero. By dualizing this lets us compute the structure of $H_*(B_k;\Z)$ and hence $H_*(F_k;\Z)$ by the Thom isomorphism theorem.
	\end{proof}
	\begin{definition}
		Let $G_j = {F_{p^j-1}}_{(p)}$ be the $p$-localisation of $F_{p^j-1}$ at a prime number $p$.
	\end{definition}
	
	\section{Adams spectral sequence}	
	The Adams spectral sequence and its generalizations are the main tools of the stable homotopy theory. There are many inequivalent definitions and convergence results concerning these spectral sequences. This section introduces what we shall call the Adams spectral sequence.
	
	\vspace{1em}
	
	We shall use the non-classical Adams spectral sequence based on $X(n+1)$. For completeness, let us define the Adams spectral sequence based on any ring spectrum $E$ here.
	\begin{definition}
		A non-canonical Adams resolution for $X$ based on $E$ is the diagram
		$$\begin{tikzcd}
		X = X_0 \arrow[d] & X_1 \arrow[d] \arrow[l] & X_2 \arrow[d] \arrow[l] & \cdots \arrow[l]\\
		K_0 & K_1 & K_2
		\end{tikzcd}$$
		in which each $X_{s+1} \to X_s \to K_s$ is a homotopy fibre sequence and $K_s$ and $E \land X_s$ are retracts of $E \land K_s$.
	\end{definition}
	Each homotopy fibre sequence $X_{s+1} \to X_s \to K_s \to \Sigma X_{s+1}$ gives a long exact sequence of homotopy groups. These comprise an exact couple and the spectral sequence associated to this exact couple is called the Adams spectral sequence for $X$ based on $E$.
	
	\vspace{1em}
	
	Under certain technical conditions (\emph{i.e.} if $E$ is flat) the $E_2$-term of the spectral sequence can be identified as a certain $\Ext$ group. 
	\begin{remark}
		No knowledge of this is required to follow the proof of the nilpotence theorem presented in this essay. This is because the information we shall need to extract from our Adams spectral sequence is very coarse -- so much so, that it can be obtained using only the $E_1$-page and the connectivity properties of $K_s$.
	\end{remark}
	The Adams spectral sequence converges to $\pi_*(X)$ under certain technical conditions. The groups $E_\infty^{s, t}(X)$ are the subquotients of $\pi_{t-s}(X)$ associated to the Adams filtration of $\pi_{t-s}(X)$.
	\begin{definition}
		Let $\alpha: S^d \to X$ be a map of spectra. The map $\alpha \in \pi_d(X)$ has Adams filtration $s$ if $s$ is the smallest integer such that $\alpha$ can be factored as
		$$S^d \xrightarrow{\alpha_1} W_1 \xrightarrow{\alpha_2} \cdots \xrightarrow{\alpha_{s-1}}W_{s-1}\xrightarrow{\alpha_s} X $$
		where $E_*(\alpha_i)=0$ for each $i$. If there is no such integer $s$, then $\alpha$ has Adams filtration $0$.
	\end{definition}
	Note that if $E_*(\alpha) \neq 0$, then the Adams filtration of $\alpha$ is $0$.
	\begin{definition}
		Let $F^s_d = \{ \alpha \in \pi_d(X) \ | \ \alpha \text{ has Adams filtration} \geq s \}$. Then the	filtration
		$$ \dots \subset F^2_d \subset F^1_d \subset F^0_d$$
		is the Adams filtration of $\pi_d(X)$.
	\end{definition}
	We can now give a more precise, although still an incomplete statement of the convergence theorem.
	\begin{theorem}\label{convergence}
		The Adams spectral sequences for $X$ based on a ring spectrum $E$ considered in this essay converge to $\pi_*(X)$. This means that
		\begin{itemize}
			\item $E_\infty^{s, t}(X) \cong \frac{F^s_{t-s}}{F^{s+1}_{t-s}}$ for all $s, t$ and
			\item $\bigcap_{s=0}^\infty F^s_{d} = 0$ for all $d$
		\end{itemize}
		where $F^s_d$ denotes the abelian groups in the Adams filtration of $\pi_d(X)$.
	\end{theorem}

%

	\chapter{Nilpotence Theorem}
	\section{Motivation}
	In the introduction we tried to explain how one might naturally arrive at the statement of the nilpotence theorem, starting from a very basic and concrete question in algebraic topology. However, this was historically not why the theorem was conjectured, nor is it the best way to think about it presently. Studying nilpotent self-maps may feel artificial and overly restrictive until a broader historical context is introduced. Such a context is provided by Nishida's theorem \cite{nishida1973nilpotency}.
	\begin{theorem}[Nishida's theorem]
		Every element of positive degree of $\pi_*^S$ is nilpotent.
	\end{theorem}
	We prove this result as an elementary consequence of the nilpotence theorem in the final chapter of this essay. Historically, however, Nishida's theorem preceded the nilpotence conjecture and in fact influenced its formulation and served as evidence for its truth.
	
	\vspace{1em}
	
	There are three ways to think about the ring structure on $\pi_*^S$. Let $f, g \in \pi_*^S$. Then their product can be thought of roughly as:
	\begin{enumerate}[itemsep=0pt, label=$-$]
		\item $f \circ g$, 
		\item $f \land g$ or
		\item $m(f,g)$.
	\end{enumerate}
	The three perspectives hint at the fact that generalizations on the Nishida's theorem in different directions may be possible. Indeed, for each of the perspectives, there is a corresponding version of the nilpotence theorem.
	
	\vspace{1em}
	The first perspective emphasizes the study of self-maps $f: \Sigma^d S \to S$ and their suspensions.
	\begin{theorem}[Nilpotence theorem, self-map form]
		Let $X$ be a finite spectrum and let $\alpha: \Sigma^d X \to X$ be a self-map for some $d$. If $MU_*(\alpha)=0$ then $\alpha$ is nilpotent.
	\end{theorem}
	The second perspective emphasizes that the product in $\pi_*^S$ comes from the smash product.
	\begin{theorem}[Nilpotence theorem, smash product form]
		Let $F$ be a finite spectrum and $f: F \to X$ a map of spectra. If $\id_{MU} \land f$ is null-homotopic, then $f$ is smash nilpotent.
	\end{theorem}
	The third perspective emphasizes the fact that $S$ has the structure of a ring spectrum with the abstract multiplication map denoted by $m$.
	\begin{theorem}[Nilpotence theorem, ring spectrum form]
		Let $R$ be a ring spectrum and let
		$$h: \pi_*(R) \to MU_*(R)$$
		be the Hurewicz homomorphism. Then every element of $\ker h$ is nilpotent.
	\end{theorem}
	We first prove the ring spectrum form of the nilpotence theorem following the Ravenel's sketch of the proof \cite{ravenel1992nilpotence}, which in turn is mostly based on the original paper by Devinatz--Hopkins--Smith \cite{devinatz1988nilpotence}. Afterwards we deduce the other two forms of the theorem and give some of its applications.
	
	\begin{remark}
		In the literature, one may find different variations of each of these forms of the nilpotence theorem. This is because the assumptions on the original spectra (e.g. connectivity, finiteness or finite type) are sometimes imposed for the ease of the exposition and sometimes omitted for generality. We prove the strongest versions of the ring spectrum and smash product forms, but a weaker version of the self-map form for its nice classical interpretation.
	\end{remark}
	
	\section{Organization of the proof}
	This section describes the structure of this chapter. As mentioned, there are several versions and variants of the nilpotence theorem and if we are not very explicit, a confusion can arise as to which one we have in mind at any particular moment. Therefore, let us be explicit about it. Most of this essay is dedicated to proving the following theorem and whenever we refer to the `nilpotence theorem', it is likely that we mean this version.
	\begin{theorem}[Nilpotence theorem, weak ring spectrum form]\label{weakRingSpectrumForm}
		Let $R$ be a connective ring spectrum of finite type and let
		$$h: \pi_*(R) \to MU_*(R)$$
		be the Hurewicz homomorphism. Then every element of $\ker h$ is nilpotent.
	\end{theorem}
	Other forms of the nilpotence theorem (smash product form, strong ring spectrum form, self-map form) are deduced from Theorem \ref{weakRingSpectrumForm} in Sections 3.5, 3.6 and 3.7. The following schematic diagram depicts the implications we prove in these sections.
	\begin{center}
		\includegraphics[width=0.8\textwidth]{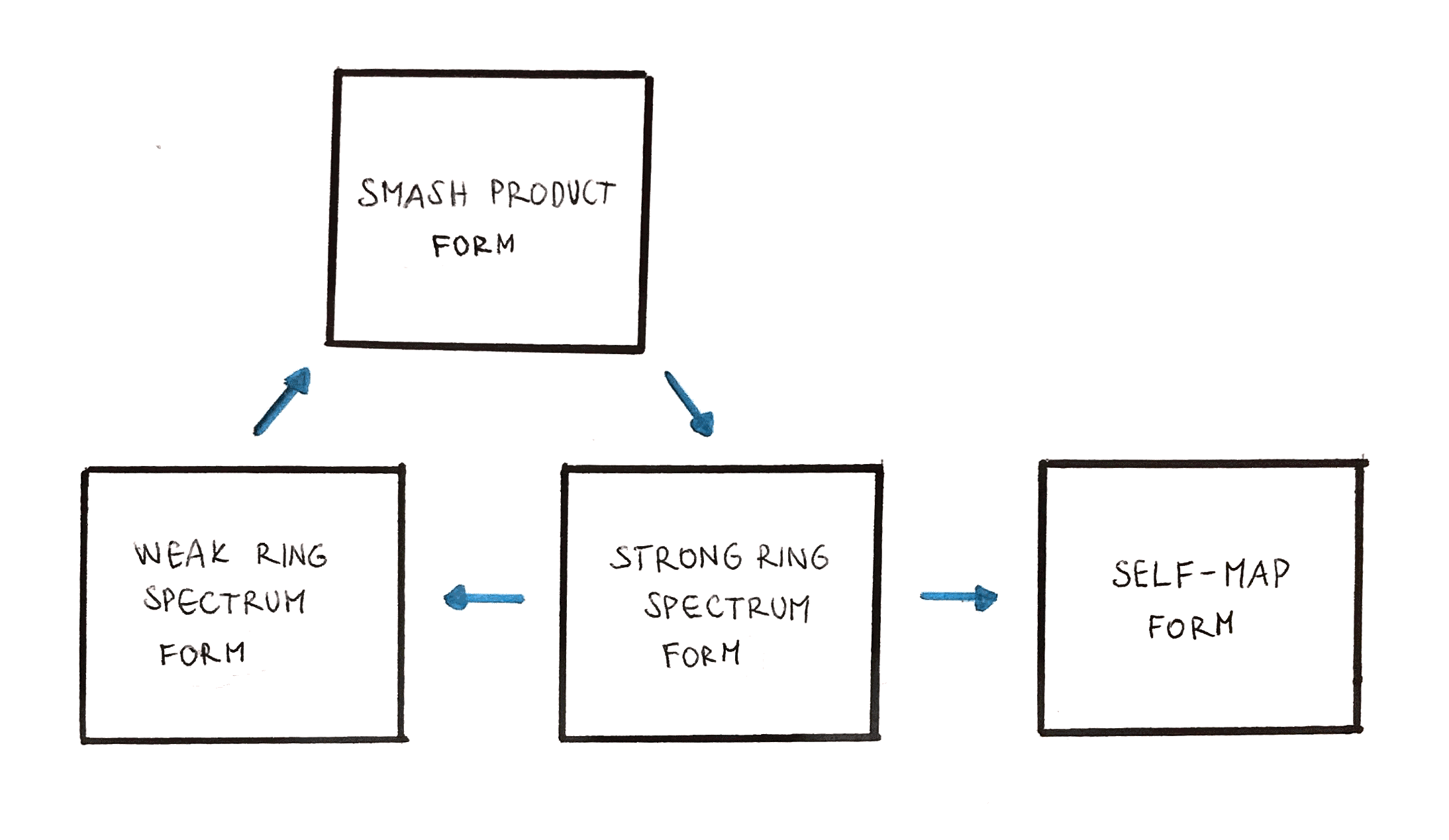}
	\end{center}
	
	\vspace{-2em}
	
	We now give an overview of the initial strategy of tackling the problem. It turned out to be unsuccessful, but it did serve as a basis for more sophisticated attempts, including the eventual proof.
	
	\subsection{Early Attempts}\label{earlyAttempts}
	Early attempts at the proof of the nilpotence theorem tried to establish that the spectra $S$ and $MU$ are Bousfield equivalent, see \cite[Section 7.4]{ravenel1992nilpotence}. A rough structure of the proposed proof was as follows. Let $\alpha \in \ker h$.
	\begin{itemize}
		\item {\bf Step I}: Show that $MU \land \alpha^{-1}R$ is contractible.
		\item {\bf Step II}: Show that $\langle S \rangle = \langle MU \rangle$.
	\end{itemize}
	By the Bousfield equivalence it now follows that $S \land \alpha^{-1}R = \alpha^{-1}R$ is contractible. By definition $\alpha^{-1}R$ is the homotopy colimit of
	$$ R \xrightarrow{\alpha} \Sigma^{-d}R \xrightarrow{\alpha} \Sigma^{-2d} R \to \cdots$$
	Taking homotopy groups of the diagram and using the fact that $\pi_*(\alpha^{-1}R)=0$ shows that every element $\beta \in \pi_*(R)$ satisfies $\alpha^m \beta =0$ for some $m$. In particular, this is true for $\alpha \in \pi_d(R)$ and so $\alpha$ is nilpotent as required.
	
	\vspace{1em}
	
	Step I of the proposed proof is true and clear. On the other hand, Step II was shown not to hold by Brown and Comenetz \cite{brown1976pontrjagin} who constructed a non-contractible spectrum $cY$ with $MU_*(cY)=0$. In fact, the spectrum $S$ turned out to live much higher than $MU$ in the Bousfield hierarchy of spectra.
	
	\subsection{Large scale structure of the proof}
	The intuitive reason for why the proposed proof structure from the previous section did not yield an actual proof is because the leap from $S$ to $MU$ is too large. Instead, we resort to  an infinite family of spectra $X(n)$ with the intention to interpolate between $S$ and $MU$. These spectra have the property that $X(1)=S$ and $X(\infty)=MU$ so they subdivide the giant leap into infinitely many smaller steps from $X(n+1)$ to $X(n)$.
	
	\vspace{1em}
	
	We will be performing ``downwards induction'': provided $X(n+1)_*(\alpha)$ is nilpotent we show that $X(n)_*(\alpha)$ is also nilpotent. This lets us eventually descend from $X(\infty)_*(R) = MU_*(R)$ to $X(1)_*(R)=\pi_*(R)$ and thus show that the original element $\alpha \in \pi_*(R)$ is nilpotent.	
	
	\vspace{1em}
		
	The inductive step is the challenging part of the proof. It is performed $p$-locally, one prime at a time. To descend from $X(n+1)$ to $X(n)$ we make transfers from $X(n+1)_{(p)}$ to $X(n)_{(p)}$ for each prime $p$. To that end we have constructed a further infinite family of spectra $G_j$ with the property that $G_0=X(n)_{(p)}$ and $G_\infty=X(n+1)_{(p)}$ and whose purpose is to interpolate between $X(n)_{(p)}$ and $X(n+1)_{(p)}$. To make this more precise, we state two crucial properties they possess.
	\begin{enumerate}[label=]
		\item {\bf Step I}: Let $\alpha \in \ker h$. If ${X(n+1)_{(p)}}_*(\alpha)$ is nilpotent, then $G_j \land \alpha^{-1}R$ is contractible for sufficiently large $j$.
		
		This step describes how the spectra $G_j$ approximate $X(n+1)_{(p)}$.
		
		\item {\bf Step II}: Show that $\langle G_j \rangle = \langle G_{j+1} \rangle$ for each $j \in \N_0$.
		
		In particular, this means that $\langle G_j \rangle = \langle X(n)_{(p)} \rangle$ so this step describes how the spectra $G_j$ approximate $X(n)_{(p)}$.
	\end{enumerate}

	\begin{remark}
		Notice the resemblance between this and the attempt in Section \ref{earlyAttempts}: the same idea is replicated on a smaller scale. Whereas we originally hoped to resolve the problem (\emph{i.e.} pass from $S$ to $MU$) using these two steps, they will now be used to complete the inductive step (\emph{i.e.} pass from $X(n+1)$ to $X(n)$).
	\end{remark}
	
	Most of the content of the nilpotence theorem is in the proof of these two properties of the interpolating spectra $G_j$. Once they are established, the inductive step can be completed by an elementary observation reminiscent of the one given in \ref{earlyAttempts}. It will be written out in full in section \ref{conclusion}.
	
	\begin{remark}
		In the rest of the essay, we deal almost exclusively with $p$-local spectra and we omit the subscript $_{(p)}$ for legibility.
	\end{remark}
	
	\begin{remark}
		We have chosen the inductive step to be $n+1 \to n$ instead of a seemingly more natural $n \to n-1$ to maintain consistency with the literature and enhance the legibility of the essay.
	\end{remark}
	
	\section{Proof}
	For the entire proof we fix $\alpha \in \ker h$. Say $\alpha \in \pi_d(R)$. We show that $\alpha$ is nilpotent.
	
	\vspace{1em}
	
	Because $\alpha$ is nilpotent iff $\alpha^m$ is nilpotent one may without loss of generality replace $\alpha$ by any of its powers $\alpha^m$. This observation will be very convenient at various stages of the proof.
	

\subsection{Preliminary lemma}
We begin the proof with a preliminary lemma in which we exhibit a certain fibration. This fibration is used extensively in Step II of the proof to relate the Bousfield equivalence classes $\langle G_j \rangle$ and $\langle G_{j+1} \rangle$, and the proof of the following lemma is is also required in Step I. 

\begin{lemma}\label{OGfibration}
	There is a fibration $B_{p^j-1} \to B_{p^{j+1}-1} \xrightarrow{q} J_{p-1}S^{2p^jn}$.
\end{lemma}
We construct this fibration as a pullback of another fibration
$$B_{p^j-1} \to \Omega SU(n+1) \xrightarrow{h} \Omega S^{2p^jn+1}$$
along the inclusion map $J_{p-1}S^{2p^jn} \hookrightarrow JS^{2p^jn} \simeq \Omega S^{2p^jn+1}$. This may not be the most direct way of establishing the lemma, but both of these fibrations together with their relationship are important later in the proof.

\begin{proof}
	Recall the following commutative diagram from Section 2.5 obtained by looping the fibration $SU(n) \to SU(n+1) \xrightarrow{e} S^{2n+1}$ and taking its pullback.
	$$
	\begin{tikzcd}
	\Omega SU(n) \arrow[r, equal, "\id"] \arrow[d] & \Omega SU(n) \arrow[d]\\
	B_k = i^* \Omega SU(n+1) \arrow[d] \arrow[r, "i_{-1,0}"] & \Omega SU(n+1) \arrow[d, "\Omega e"]\\
	J_kS^{2n} \arrow[r, hook, "i"] & JS^{2n} \simeq \Omega S^{2n+1}.  
	\end{tikzcd}
	$$		
	This construction was used to define the spaces $B_k$. Specializing to the case $k=p^j-1$, we now extend the bottom two rows to homotopy fibre sequences.
	
	\vspace{1em}
	
	Let us start with the bottom row. We claim that there is a homotopy fibre sequence $J_{p^j-1}S^{2n} \to \Omega S^{2n+1} \xrightarrow{H} \Omega S^{2np^j+1}$ where $H$ is the James-Hopf map defined in \ref{JamesHopf} by taking $X=S^{2n}$. To see this, consider the homological Serre spectral sequence with coefficients in $\F_p$ associated to $H$. Both the total space and the base are of the type $\Omega S^m$ and their homology is
	$$H_i(\Omega S^m ; \F_p) \cong \begin{cases}
	\F_p & i \in \{ 0, \ m-1, \ 2(m-1), \dots \}\\
	0 & \text{otherwise}
	\end{cases}$$
	which is obtained by another standard application of the Serre spectral sequence. Let $F$ be the homotopy fibre of $H$. For purely geometrical reasons there are no differentials entering any of the groups $E_{0,q}^r = H_q(F; \F_p)$ for $q \leq 2p^jn-2$. Thus the non-trivial homology groups of $F$ in this range are $H_q(F;\F_p) \cong \F_p$ for $q \in \{ 0, 2n, \dots, 2n(p^j-1) \}$. The first potentially non-trivial differential is the transgression $E_{2p^jn,0}^{2p^jn} \to E_{2p^jn-1,0}^{2p^jn}$. However, recall that the James-Hopf map $H$ induces a surjection on (mod $p$) homology as discussed in Section 2.3. By naturality of the Serre spectral sequence, this surjection factorises as 
	$$H_{2p^jn}(\Omega S^{2n+1}; \F_p) \to E_{2p^jn,0}^\infty \to H_{2p^jn}(\Omega S^{2p^jn+1};\F_p)$$ through $E_{2p^jn,0}^\infty$ and so $E_{2p^jn,0}^\infty \cong \F_p$. It follows that the transgression differential is the zero map too.
	
	By dualizing, considering the Serre spectral sequence for cohomology and using its multiplicative structure we deduce that all differentials in the spectral sequence are the zero maps. This completely determines the homology and cohomology groups of $F$ and these are precisely those of $J_{p^j-1}S^{2n}$. Therefore
	$$ J_{p^j-1}S^{2n} \xrightarrow{i} \Omega S^{2n+1} \xrightarrow{H} \Omega S^{2p^jn+1} $$
	is a homotopy fibre sequence and hence so is
	$$B_{p^j-1} \xrightarrow{i_{-1,0}} \Omega SU(n+1) \xrightarrow{h:= H \circ \Omega e} \Omega S^{2p^jn+1}$$
	because the lower part of the diagram above is a pullback square. The entire construction is best described with the following diagram:
	$$
	\begin{tikzcd}
	\Omega SU(n) \arrow[r, equal, "\id"] \arrow[d] & \Omega SU(n) \arrow[d]&\\
	B_{p^j-1} \arrow[d] \arrow[r, "i_{-1,0}"] & \Omega SU(n+1) \arrow[d, "\Omega e"] \arrow[r, "H \circ \Omega e"] & \Omega S^{2p^jn+1} \arrow[d, equal, "\id"]\\
	J_{p^j-1}S^{2n} \arrow[r, hook] & JS^{2n} \simeq \Omega S^{2n+1} \arrow[r, "H"]& \Omega S^{2p^jn+1}.  
	\end{tikzcd}
	$$

	
	We have shown that the middle row of this diagram is a fibration. The fibration in the statement of this lemma will be constructed by taking its pullback along $J_{p-1}S^{2p^jn} \hookrightarrow JS^{2p^jn} \simeq \Omega S^{2p^jn+1}$. 
	
	\vspace{1em}
	
	The rest of the argument has a very similar character to what we have already seen. In fact, this is more than a coincidental similarity; we literally apply the same construction to the diagram
	$$
	\begin{tikzcd}
	B_{p^j-1} \arrow[r, equal, "\id"] \arrow[d]& B_{p^j-1} \arrow[d]\\
	F \arrow[d] \arrow[r]& \Omega SU(n+1) \arrow[d, "h"]\\
	J_{p-1}S^{2p^jn} \arrow[r]& \Omega S^{2p^jn+1}   
	\end{tikzcd}
	$$
	where $F$ is now the pullback that we wish to describe more explicitly.
	
	\vspace{1 em}
	
	Let $H': \Omega S^{2p^jn+1} \to \Omega S^{2p^{j+1}n+1}$ be a James-Hopf map. It induces a surjection on (mod $p$) homology and by an analogous analysis of the Serre spectral sequence for (mod $p$) homology and cohomology we obtain that
	$$ J_{p-1}S^{2p^jn} \to \Omega S^{2p^jn+1} \xrightarrow{H'} \Omega S^{2p^{j+1}n+1}$$
	is a homotopy fibre sequence. As earlier because $F$ is the pullback
	$$ F \to \Omega SU(n+1) \xrightarrow{H' \circ H \circ \Omega e} \Omega S^{2p^{j+1}n+1} $$
	is a homotopy fibre sequence too. So it can be identified with the fibration
	$$ B_{p^{j+1}-1} \to \Omega SU(n+1) \to \Omega S^{2p^{j+1}n+1}$$
	obtained in the first part of this argument. The fibration
	$$ B_{p^j-1} \to B_{p^{j+1}-1} \xrightarrow{q} J_{p-1}S^{2p^jn}$$
	follows as required.
\end{proof}
\begin{remark}
	This remark is best read once one is roughly comfortable with the proof of the nilpotence theorem. We have not only constructed the fibration required by Lemma \ref{OGfibration}. In addition, we have exhibited a larger diagram
	$$
	\begin{tikzcd}
	B_{p^j-1} \arrow[r, equal, "\id"] \arrow[d] & B_{p^j-1} \arrow[d]\\
	B_{p^{j+1}-1} \arrow[d, "q"] \arrow[r]& \Omega SU(n+1) \arrow[d, "h"]\\
	J_{p-1}S^{2p^jn} \arrow[r]& JS^{2p^jn}   
	\end{tikzcd}
	$$
	in which the base spaces are the James construction and the partial James construction. Let us share some informal intuition about the role of this diagram in the proof.
	
	I like to think of the spectra $X(n) = G_0, G_1, \dots, G_\infty = X(n+1)$ as points in a topological space homeomorphic to a $1$-point compactification of $\N$. Taking the Bousfield equivalence classes is then a function $\langle \cdot \rangle$ into a totally ordered discrete set. This function is locally constant since $\langle G_j \rangle = \langle G_{j+1} \rangle$ (this is proven in Step II), but it is not globally constant since $\langle G_j \rangle < \langle G_\infty \rangle$. The point of discontinuity is represented by the digram above.
	
	The claim that $\langle G_j \rangle = \langle G_{j+1} \rangle$ is proven using induction by considering the James filtration of the partial James construction $J_{p-1}S^{2p^jn}$. This fails for $JS^{2p^jn}$ since the James filtration contains infinitely many terms and induction can no longer be applied. As a result of this failure we have $\langle X(n) \rangle < \langle X(n+1) \rangle$.
\end{remark} 

\subsection{Step I}
	\begin{proposition}\label{StepI}
		Let $X(n+1)_*(\alpha)$ be nilpotent. Then $G_j \land \alpha^{-1}R$ is contractible for sufficiently large $j$.
	\end{proposition}
	In this section, we prove the first crucial ingredient of the proof. We have mentioned that the purpose of Steps I and II is to descend from $X(n+1)$ to $X(n)$. A sequence of intermediate spectra $G_j$ was constructed to interpolate between the two. Proceeding in this spirit, we can describe the content of the Proposition \ref{StepI} as:
	\begin{center}
		\ref{StepI} shows how the $G_j$ approximate $X(n+1)$.
	\end{center}
	Dually, Step II will describe a way in which the $G_j$ also approximate $X(n)$.
	
	\vspace{1 em}
	
	To show that $G_j \land \alpha^{-1}R$ is contractible, it is sufficient to establish that its homotopy groups vanish and this is what we shall do. This is most naturally done by studying suitable Adams spectral sequences for $\pi_*(G_j)$, $\pi_*(R)$ and thus for $\pi_*(G_j \land R)$.
	
	The proof presented in this essay has a geometric flavour. Following Ravenel \cite[Section 9.2]{ravenel1992nilpotence}
	we construct the Adams resolution for $G_j$ through explicit cobfibre sequences of spaces. Studying their connectivity allows us to establish the existence of vanishing lines with arbitrarily small slopes in the Adams spectral sequence for $\pi_*(G_j)$. Using the connectivity of $R$, the conclusion follows quickly thereafter.
	\begin{proof} 
		We follow this list of steps.
		\begin{enumerate}[itemsep=0pt, label=$\textbf{-- \arabic* --}$]
			\item Construct a non-canonical Adams resolution for $G_j$ based on $X(n+1)$ of the form
			$$\begin{tikzcd}
			G_j = X_0 \arrow[d, "f_0"] & X_1 \arrow[d, "f_1"] \arrow[l] & X_2 \arrow[d, "f_2"] \arrow[l] & \cdots \arrow[l]\\
			K_0 & K_1 & K_2
			\end{tikzcd}.$$
			\item Show that $K_s$ is $(2p^jn-1)s$-connected.
			\item Using the first two points and the connectivity of $R$, establish vanishing lines with slopes $\frac{1}{2p^jn-1}$ in the Adams spectral sequence for $\pi_*(G_j \land R)$.
			\item Conclude that $G_j \land \alpha^{-1}R$ is contractible.
		\end{enumerate}
		\begin{center}
			\textbf{-- 1 --}
		\end{center}
		We begin with an inductive construction of a non-canonical Adams resolution for $G_j$ based on $X(n+1)$. Let it be of the form
		$$\begin{tikzcd}
		G_j = X_0 \arrow[d, "f_0"] & X_1 \arrow[d, "f_1"] \arrow[l] & X_2 \arrow[d, "f_2"] \arrow[l] & \cdots \arrow[l]\\
		K_0 & K_1 & K_2
		\end{tikzcd}.$$
		Recall the fibration $B_{p^j-1} \to \Omega SU(n+1) \xrightarrow{h} \Omega S^{2p^jn+1}$ constructed in Lemma \ref{OGfibration}. Thomification yields a map of spectra $f_0: G_j \to X(n+1)$ and let this be the map $f_0$ in the Adams resolution. This defines $X_1$ as the cofibre of $f_0$ and starts the induction. The maps $f_s$ for $s \geq 1$ will be constructed by first exhibiting cofibre sequences of spaces
		$$ Y_s \to L_s \to Y_{s+1}$$
		which are used to obtain fibre sequences of spectra
		$$ X_{s+1} \to X_s \xrightarrow{f_s} K_s$$
		after Thomifying and taking suitable suspensions.
		
		\vspace{1em}
		
		To motivate the next definition, we remark that the maps $i_{s, t}$ for suitable $s$ and $t$ will be used to define the cofibrations $Y_s \to L_s$ we aim to construct. Let $B=\Omega S^{2p^jn+1}$ and $* \in B$ be a basepoint. For any $s \in \N_0$ define $H_s = * \times B \times \dots \times B \times \Omega SU(n+1)$ where there are $s$ factors of $B$. For $t \in \{  0, \dots, s+1  \}$ define maps $i_{s, t}: H_{s} \to H_{s+1}$ by
		$$ i_{s,t}(*,b_1, \dots, b_s, e) =
		\begin{cases}
		(*,*, b_1, \dots, b_s, e) & | \quad t=0\\
		(*,b_1, \dots, b_t, b_t, b_{t+1}, \dots, b_s, e) & | \quad t \in \{ 1, \dots, s \} \\
		(*,b_1, \dots, b_s, h(e), e) & | \quad t = s+1.
		\end{cases} $$
		Less rigorously, but more intuitively, one can think of the maps $i_{s,t}$ as trying to double the coordinate $b_t$. This can mostly be done, but there is a problem when $t=s+1$, because the $b_{s+1}$ does not exist. A slight modification leads to the definition of the maps $i_{s,t}$.
		
		Now define the spaces
		\begin{align*}
		Y_s &= \frac{H_{s-1}}{\im i_{s-2,0} \cup \dots \cup \im i_{s-2, s-1}}\\
		L_s &= \frac{H_s}{\im i_{s-1, 0} \cup \dots \cup \im i_{s-1, s-1}}.
		\end{align*}
		Instead of trying to decipher the meaning of the indices in the expressions above, it may be easier to use the following intuitive characterisation of the collapsed subspaces.
		\begin{itemize}
			\item In both cases the collapsed subspace contains the points in which some two consecutive coordinates are the same.
		\end{itemize}
		The collapsed subspace in the definition of $L_s$ is completely determined by these two properties. For $Y_s$ we additionally have:
		\begin{itemize}
			\item The collapsed subspace in the definition of $Y_s$ contains the points with $b_{s-1}=h(e)$ \emph{i.e.} points of the form $(*,b_1, \dots, b_{s-2}, h(e), e)$.
		\end{itemize}
		Henceforth, the map $i_{s-1, s}: H_{s-1} \to H_s$ for $s \geq 1$ induces a well-defined map
		\begin{align*}
		Y_s &\to L_s\\
		(*,b_1, \dots, b_{s-1}, e) &\mapsto (*,b_1, \dots, b_{s-1}, h(e), e).
		\end{align*}
		To verify that this is indeed well-defined, we can use the characterisation of the collapsed subspaces. Let $b \in Y_s$ be a point in the collapsed subspace of $Y_s$. If $b$ satisfies the first condition, its image clearly does too. If $b$ satisfies the last condition, then $b_{s-1}=h(e)$ and the image of $b$ has two consecutive coordinates that are the same. Hence in either case the image of $b$ is collapsed in $L_s$ and the map is well defined.
		
		\vspace{1em}
		
		The map $f_s: Y_s \to L_s$ is clearly injective. Its cofibre is 
		$$\frac{L_s}{\im_{s-1,s}} = \frac{H_s}{\im i_{s-1,0} \cup \dots \cup \im i_{s-1, s}}= Y_{s+1}.$$
		Therefore we have constructed cofibre sequences of spaces $Y_s \to L_s \to Y_{s+1}$.
		
		
		\vspace{1em}
		
		The next step is to Thomify these cofibre sequences into cofibre sequences of spectra and thus construct an Adams resolution of $G_j$. To do this, consider the projection maps $p_s: H_s \to \Omega SU(n+1) \hookrightarrow BU$ given via $(b_1, \dots, b_s, e) \mapsto e$ making $H_s$, $Y_s$ and $L_s$ into objects of ${\bf Top}_{BU}$. Moreover, it is clear that the maps $i_{s-1,s}$ and their restrictions $Y_s \to L_s$ are morphisms in ${\bf Top}_{BU}$. Therefore we can Thomify the construction to obtain a cofibre sequence of spectra
		$$ Y_s^{p_{s-1}} \xrightarrow{f_s} L_s^{p_s} \to Y_{s+1}^{p_s}$$
		as required. Letting $K_s=\Sigma^{-s} L_s^{p_s}$ and $X_s=\Sigma^{-s} Y_s^{p_{s-1}}$ we obtain homotopy fibre sequences $X_{s+1} \to X_s \xrightarrow{f_s} K_s$ of spectra and thus a non-canonical Adams resolution for $G_j$ as required.
		\begin{center}
			\textbf{-- 2 --}
		\end{center}
		We have constructed a non-canonical Adams resolution for $G_j$ based on $X(n+1)$, a development corresponding to the first point of the plan we had outlined prior to starting with the proof. There seems to be only one natural continuation: consider the Adams spectral sequence associated to this resolution. Unfortunately, we would not learn much about $G_j$ if we attempted this immediately for the resolution is sufficiently mysterious. Instead, we first establish some connectivity properties of the spectra $K_s$ and then consider the Adams spectral sequence. The connectivity information alone allows us to extract all that we need from the spectral sequence.
		
		\vspace{1 em}
		
		In order to study the connectivity of the spectra $K_s$, we calculate the structure of the graded abelian group $\widetilde{H}_*(L_s ; \F_p)$. Note the homeomorphism $L_1 = \frac{H_1}{i_{0,0}} \cong \frac{* \times B \times \Omega SU(n+1)}{* \times * \times \Omega SU(n+1)} \cong B \land \Omega SU(n+1)_+$ which yields by the Künneth theorem $\widetilde{H}_*(L_1; \F_p) \cong \widetilde{H}_*(B;\F_p) \otimes H_*(\Omega SU(n+1);\F_p)$. We shall now prove a corresponding formula for the homology of any $L_s$. Recall that these are defined defined via the cofibre sequences
		$$  \im i_{s-1, 0} \cup \dots \cup \im i_{s-1, s-1} \to H_s \to L_s  $$ 
		where the subspace $ T_s := \im i_{s-1, 0} \cup \dots \cup \im i_{s-1, s-1}$ has an explicit description
		$$T_s = \{ (b_0,b_1, \dots, b_s, e) \in H_s \ | \ b_i = b_{i+1} \text{ for some } i   \}.$$
		Note that $b_0=*$ by definition of $H_s$. We can write $T_s= \bigcup_{i=0}^{s-1} T^i$ where
		$$T^i = \{ (b_0, b_1, \dots, b_s, e) \in T_s \ | \  b_i = b_{i+1}   \}.$$
		Therefore $T_s$ can be expressed as the union of two spaces $\bigcup_{i=0}^{s-2} T^i$ and $T^{s-1}$ with the following intuitive description.
		\begin{itemize}[itemsep=0pt]
			\item The space $\bigcup_{i=0}^{s-2} T^i$ has two consecutive equal coordinates $b_i = b_{i+1}$ for some $0 \leq i \leq s-2$ and the coordinate $b_s$ can be arbitrary. Therefore $\bigcup_{i=0}^{s-2} T^i \cong T_{s-1} \times B$.
			\item The space $T^{s-1}$ satisfies $b_{s-1}=b_s$ and all other coordinates are arbitrary. So $T^{s-1} \cong H_{s-1}$. Note in particular that the inclusion $\phi_s: H_{s-1} \to T_s$ is given by $(b_0, \dots, b_{s-1}, e)\mapsto (b_0, \dots, b_{s-1}, b_{s-1}, e)$.
			\item Their intersection $\left(\bigcup_{i=0}^{s-2} T^i \right) \cap T^{s-1}$ is homeomorphic to $T_{s-1}$ by combining both of the above descriptions.
		\end{itemize}
		This gives a pushout square	
		$$
		\begin{tikzcd}
		T_{s-1} \arrow[d] \arrow[r] & H_{s-1} \arrow[d, "\phi_s"] \\
		T_{s-1} \times B \arrow[r]& T_s
		\end{tikzcd}
		$$
		and so $L_s = \frac{H_s}{T_s}$ can be obtained from $H_s$ in two steps by first collapsing $T_{s-1} \times B$ and then the rest of $T_s$. Formally, the cofibre of the inclusion $T_{s-1} \times B \to T_s$ is given by $\frac{H_s}{T_{s-1}\times B} \cong \frac{H_{s-1} \times B}{T_{s-1}\times B} \cong \frac{H_{s-1}}{T_{s-1}} \land B_+$. Under this homeomorphism, we consider the cofibre sequence
		$$  \frac{H_{s-1}}{T_{s-1}} \xrightarrow{\phi_s} \frac{H_{s-1}}{T_{s-1}} \land B_+ \to \frac{H_s}{T_s} = L_s  $$
		to collapse the rest. Since $\frac{H_{s-1}}{T_{s-1}} = L_{s-1}$ this will inductively allow us to compute the homology of $L_s$. Note from discussion above that the map $\phi_s$ is the identity on the first factor and it extracts the coordinate $b_{s-1}$ on $B_+$.
		
		\vspace{1em}
		
		As a morphism in $\bf Top$, the map $\phi_s$ is in general \emph{not} homotopic to $\id \land *_+$, \emph{nor} do they induce the same map on homology. However, they only differ by the automorphism
		$$\psi = ({\pi_1}_* - {\pi_2}_* + *_*) \circ (\phi_s \land \id_{B_+})_*$$ of $H_* \left(\frac{H_{s-1}}{T_{s-1}}\land B_+ ; \F_p\right)$ where $\pi_i: B \times B \to B$ for $i \in \{1, 2\}$ are the projection maps and $*:B \times B \to B$ is the constant map $*$. With some abuse of notation, one could say that the automorphism $\psi$ is induced by
		$$\begin{tikzcd}
		\frac{H_{s-1}}{T_{s-1}}\land B_+ \arrow[r, "\phi_s \land \id_{B_+}"] & \frac{H_{s-1}}{T_{s-1}}\land B_+  \land B_+ \arrow[rr, dashed, "\id \land (\pi_1 - \pi_2 + *)"]&& \frac{H_{s-1}}{T_{s-1}}\land B_+
		\end{tikzcd}$$
		where the dashed arrow represents that the second map is not a well-defined map in $\bf Top$ -- it is a sum of maps. However, the expression becomes well-defined upon passing to homology. An elementary calculation
		\begin{align*}
			(\psi \circ (\id \land *_+)_*)(b_0, \dots, b_{s-1},e) &= \psi (b_0, \dots, b_{s-1}, *, e)\\
			&= ((\pi_1)_* - (\pi_2)_* + *_*)(b_0, \dots, b_{s-1}, b_{s-1}, *, e)\\
			&= (b_0, \dots, b_{s-1},b_{s-1}, e)
		\end{align*}
		shows that $\psi \circ (\id \land *_+)_* = (\phi_s)_*$ so $\phi_s$ and $\id \land *_+$ only differ by an automorphism upon passing to homology.
	
		\vspace{1em}
		
		The upshot is that the long exact sequences on homology groups induced by the cofibre sequences associated to $\phi_s$ and $\id \land *_+$ are the same up to an automorphism and the latter is well-understood. Namely, it splits into short exact sequences as
		$$ 0 \to \widetilde{H}_k\left(\frac{H_{s-1}}{T_{s-1}}; \F_p\right) \to \bigoplus_{i+j=k} \widetilde{H}_i\left(\frac{H_{s-1}}{T_{s-1}}; \F_p\right) \otimes \widetilde{H}_j(B_+; \F_p) \to \widetilde{H}_k(L_s; \F_p) \to 0$$
		using the Künneth isomorphism to express the homology of $\frac{H_{s-1}}{T_{s-1}} \land B_+$ as a tensor product. It follows that
		\begin{align*}
			\widetilde{H}_*(L_s; \F_p) &\cong \widetilde{H}_*\left(\frac{H_{s-1}}{T_{s-1}} \land B; \F_p\right) \cong \widetilde{H}_*(L_{s-1}; \F_p)  \otimes \widetilde{H}_*(B; \F_p)\\
			&\cong \widetilde{H}_*(B^{\land {s-1}}; \F_p) \otimes H_*(\Omega SU(n+1) ; \F_p) \otimes \widetilde{H}_*(B; \F_p)\\
			&\cong \widetilde{H}_*(B^{\land s}; \F_p) \otimes H_*(\Omega SU(n+1) ; \F_p)
		\end{align*}
		where we use the induction hypothesis to pass to the middle row and other isomorphisms come from purely algebraic manipulations and applications of the Künneth theorem. This completes the inductive step and it follows that $\widetilde{H}_*(L_s ; \F_p) \cong \widetilde{H}_*(B^{\land s}; \F_p) \otimes  H_*(\Omega SU (n+1); \F_p)$.
		
		\vspace{1em}
		
		We have now expressed the homology of $L_s$ in a convenient form, but ultimately we would like to know about the connectivity of $K_s$. The translation between these results is obtained by a few applications of the Hurewicz's theorem and the Thom isomorphism theorem in the next paragraph.
		
		Recall that $B= \Omega S^{2p^jn+1}$. We have $\pi_i(\Omega S^{2p^jn+1}) \cong \pi_{i+1}(S^{2p^jn+1})$ from the long exact sequence of homotopy groups associated to the path fibration. The sphere $S^{2p^jn+1}$ is $(2p^jn)$-connected and it follows that $\Omega S^{2p^jn+1}$ is $(2p^jn-1)$-connected. By Hurewicz's theorem the reduced homology groups $\widetilde{H}_i(\Omega S^{2p^jn+1};\Z)$ vanish for $i \leq 2p^jn-1$ too. Using the Universal coefficients theorem and the Künneth theorem yet again we obtain 
		$$\widetilde{H}_i(B^{\land s} ; \F_p) \cong \bigoplus_{i_1 + \dots + i_s = i} \widetilde{H}_{i_1}(B ; \F_p) \otimes \dots \otimes  \widetilde{H}_{i_s}(B; \F_p).$$
		Since we are working $p$-locally, the Hurewicz's theorem gives that the space $B^{\land s}$ is $(2p^jn-1)s$-connected. Therefore $L_s$ is $(2p^jn-1)s$-connected too and finally so is $K_s$ by the Thom isomorphism theorem.
		\begin{center}
			\textbf{-- 3 --}
		\end{center}
		With this connectivity information at our disposal, we may now study the Adams spectral sequence. The Adams resolution for $G_j$ can be used to construct a resolution for $G_j \land R$. By smashing every spectrum and map by $R$ we obtain the diagram
		$$\begin{tikzcd}
		G_j \land R= X_0 \land R\arrow[d, "f_0 \land \id_R"] & X_1 \land R\arrow[d, "f_1 \land \id_R"] \arrow[l] & X_2 \land R\arrow[d, "f_2 \land \id_R"] \arrow[l] & \cdots \arrow[l]\\
		K_0 \land R& K_1 \land R& K_2 \land R
		\end{tikzcd}$$
		in which $X_{s+1} \land R \to X_s \land R \to K_s \land R$ is a homotopy fibre sequence for each $s$. This is a non-canonical Adams resolution for $G_j \land R$ based on $X(n+1)$. Each of these cofibre sequences induces a long exact sequence of homotopy groups
		$$ \cdots \to \pi_i(X_{s+1} \land R) \to \pi_i(X_s \land R) \to \pi_i(K_s \land R) \to \cdots $$
		and these constitute an exact couple. The corresponding spectral sequence is the Adams spectral sequence converging to $\pi_*(G_j \land R)$ based on $X(n+1)$.
		
		\vspace{1em}
		
		The $E_1$-page of this spectral sequence has $E_1^{s,t}(G_j \land R) = \pi_{t-s}(K_s \land R)$. Recall that $K_s$ is $(2sp^jn-s)$-connected and by the initial assumption of the nilpotence theorem $R$ is connective \emph{i.e.} $N$-connected for some $N \in \Z$. Combining the two connectivity results we obtain that $G_j \land R$ is $(2sp^jn-s+N)$-connected. Therefore
		$$E_1^{s, t}(G_j \land R) = 0 \quad \text{for} \quad t-s \leq (2p^jn-1)s+N.$$
		It follows that the same is true on the $E_2$-page of the spectral sequence:
		$$E_2^{s, t}(G_j \land R) = 0 \quad \text{for} \quad t-s \leq (2p^jn-1)s+N.$$
		
		We can represent the groups of the $E_2$-page of the spectral sequence as lattice points in the plane. Using the coordinate system with coordinates $(t-s, s)$ the paragraph above shows that the groups $E_2^{s, t}(G_j \land R)$ vanish above the line
		$$ s= \frac{1}{2p^jn-1}(t-s)-\frac{N}{2p^jn-1}$$
		with the slope $\frac{1}{2p^jn-1}$. The crucial property of this family of vanishing lines is that they can be made arbitrarily close to a horizontal line: their slope tends to $0$ as $j \to \infty$. This concludes the third point of the plan.
		\begin{center}
			\textbf{-- 4 --}
		\end{center}
		Notice that the discussion so far was centred around the properties of the spectral sequences of $\pi_*(G_j)$ and $\pi_*(G_j \land R)$ and the element $\alpha \in \pi_d(R)$ has not yet played a role. Since we are ultimately interested in the homotopy groups $\pi_*(G_j \land \alpha^{-1}R)$, let us consider $\alpha$. One of the assumptions of this proposition is that $X(n+1)_*(\alpha)$ is nilpotent. We have justified in $3.3$ that we may without loss of generality at any point replace $\alpha$ by any of its powers $\alpha^m$. We do so at this point to make $X(n+1)_*(\alpha)=0$. Therefore the Adams filtration of $\alpha$ is $1$.
		 
		 \vspace{1em}
		
		The Adams spectral sequence for $R$ based on $X(n+1)$ converges to the homotopy groups of $R$ and $F^1_d/F^2_d \cong E_\infty^{1, d+1}(R)=\cap_{r > 1} E^{1, d+1}_r(R)$. Let $\widehat{\alpha} \in E_2^{1, d+1}(R)$ be an element on the $E_2$-page detecting $\alpha \in F^1_d \subset \pi_d(R)$. Then $\widehat{\alpha}$ lies on the line through the origin with a slope of $\frac{1}{d}$ in the Adams spectral sequence for $\pi_*(R)$. Let us now fix $j \in \N$ sufficiently large so that $\frac{1}{2p^jn-1}< \frac{1}{d}$.

		
		\vspace{1em}

		Moving to the spectral sequence for $\pi_*(G_j \land R)$, let $\beta \in \pi_*(G_j \land R)$ be arbitrary. The map of spectra $R \to G_j \land R$ induces an action of $\pi_*(R)$ on $\pi_*(G_j \land R)$. Let $\beta$ be detected by $\beta' \in E^{u,v}_2(G_j\land R)$ and $\id_{G_j} \land \alpha \in \pi_d(G_j \land R)$ be detected by $\alpha' \in E_2^{1, d+1}(G_j \land R)$. If $\beta \alpha^m \neq 0$, then it is detected by an element
		$$\beta' {\alpha'}^m \in E^{u+m, v+md}_2(G_j \land R).$$		
		For sufficiently large $m$, the point $(u+m, v+md)$ lies above the vanishing line, since the slope of the vanishing line is $\frac{1}{2p^jn-1}< \frac{1}{d}$ by the choice of $j$. It follows that $E_2^{u+m, v+md}(G_j \land R)$ vanishes and in particular $\beta \alpha^m = 0$ in $\pi_*(G_j \land R)$. Therefore $\beta=0$ in $\pi_*(G_j \land \alpha^{-1}R)$. Because $\beta$ was arbitrary it follows that $\pi_*(G_j \land \alpha^{-1}R)=0$ and so $G_j \land \alpha^{-1}R$ is contractible  for sufficiently large $j$ as required.
	\end{proof}
	\begin{remark}
		The only reason for passing to the $E_2$-page of the spectral sequence in this argument is that the $E_2$-page and all subsequent pages are independent of the choice of the Adams resolution. This simplifies the exposition linguistically: we can refer to `the $E_2$-page' instead of to `the $E_1$-page associated to our non-canonical Adams resolution'.
	\end{remark}
	\begin{remark}
		The nilpotence theorem can be rephrased by saying that the $E_\infty$-page of the Adams-Novikov spectral sequence for $\pi_*(R)$ (based on $X(n)$ or $MU$) contains vanishing lines of arbitrarily small slopes. It then follows exactly as in the above proof that every $\alpha \in \pi_*(R)$ detected by the Hurewicz homomorphism is nilpotent.
		
		This is why we emphasize that the above proof gives vanishing lines on the $E_2$-page of the spectral sequence for $\pi_*(G_j \land R)$, but it says nothing about the $E_2$-page of the spectral sequence for $\pi_*(R)$. Indeed, in this spectral sequence such vanishing lines \emph{only} exist at the $E_\infty$-page and it is very hard to show that they do. An entire Step II is dedicated to establishing the machinery powerful enough to leverage the information from $\pi_*(G_j \land R)$ to $\pi_*(R)$.
	\end{remark}

\subsection{Step II}
	\begin{proposition}\label{G_j=G_j+1}
		$\langle G_j \rangle = \langle G_{j+1} \rangle$ for any $j$.
	\end{proposition}
	In this section we prove the second crucial proposition. Since $G_0=X(n)$, the claim can be rephrased by saying that $\langle G_j \rangle = \langle X(n) \rangle$ for all $j$. This completes our intuitive description of the inductive step:
	\begin{center}
		\ref{G_j=G_j+1} shows how the $G_j$ approximate $X(n)$.
	\end{center}

	This is the deepest, conceptually the most complex part of the nilpotence theorem. It involves a lot of technically detailed calculations with very specific spectra. In order to facilitate the understanding of the proof, we outline the main intermediate steps before delving into rigour. We follow the plan:
	\begin{enumerate}
		\item Define a map $b: G_j \to \Sigma^{-|b|}G_j$.
		\item Prove that $\langle G_j \rangle  = \langle G_{j+1} \rangle \vee \langle b^{-1}G_j \rangle$.
		\item Prove that $b^{-1}G_j$ is contractible.
	\end{enumerate}
	Proposition \ref{G_j=G_j+1} follows immediately from these three claims, each of which is involved by itself. To give an overview of the large scale proof structure we propose the following substeps:
	\begin{enumerate}
		\item Define a map $b: G_j \to \Sigma^{-|b|}G_j$.
		\subitem (i) Construct a filtration $B_{p^j-1} = E_0  \subset \dots \subset E_{p-1} = B_{p^{j+1}-1}$.
		\subitem (ii) Establish equivalences $\theta_k: E_k^\xi / E_{k-1}^\xi \to \Sigma^{2mk} E_0^\xi$.
		\subitem (iii) Use $\theta_k$ and the associated cofibre sequences to define $b$.
		\item Prove that $\langle G_j \rangle  = \langle G_{j+1} \rangle \vee \langle b^{-1}G_j \rangle$.
		\item Prove that $b^{-1}G_j$ is contractible.
		\subitem (i) Establish the factorization of $b$ via
		$$
		\begin{tikzcd}
		\Sigma^{2mp-2}G_j \arrow[d, equal, "\id"] \arrow[r, "b"] & G_j \\
		S^{2mp-2} \land G_j \arrow[r, "\gamma \land \id"]& \Sigma^\infty \Omega^2 S^{2m+1}_+ \land G_j \arrow[u, "\mu"]
		\end{tikzcd}
		$$
		\subitem (ii) Using the properties of the Snaith's splitting structure of $\Omega^2 S^{2m+1}_+$ and the diagram in (i) show that $\id_{b^{-1} G_j}$ factorizes as $b^{-1}G_j \to H\F_p \land G_j \to b^{-1} G_j$.
		
		\subitem (iii) Prove that ${H\F_p}_*(b)=0$ and hence that $H\F_p \land G_j$ is contractible.
	\end{enumerate}
	We believe this list of steps could serve as a useful reference worth revisiting if one loses the big picture while studying the details. It is intended as a map guiding the reader through the proof.
	
	
	
	This section is divided into subsections corresponding to the main steps of the plan. Each of the subsections begins with an outline of the step followed by a complete proof.
	
	\subsubsection{Construction of $b: G_j \to \Sigma^{-|b|}G_j$}
	The aim of this subsection is to construct a map $b: G_j \to \Sigma^{-|b|}G_j$. This is the central object of the entire nilpotence theorem and most of the proof is dedicated to establishing its definition, alternative formulations and factorizations. The purpose of $b$ is to relate the Bousfield equivalence classes of $G_j$ and $G_{j+1}$.
	
	\vspace{1em}
	
	Recall that $G_j$ and $G_{j+1}$ are already related for they are the $p$-localisations of the Thom spectra associated to $B_{p^j-1}$ and $B_{p^{j+1}-1}$. There is a natural injection $B_{p^j-1} \hookrightarrow B_{p^{j+1}-1}$ inducing a map of spectra $G_j \to G_{j+1}$. One could consider the cofibre of this map, relate it to $G_j$ and then associate $b$ to this setting.
	
	\vspace{1em}
	
	However, there is a good reason to construct $b$ in a greater generality, and we shall do that. The reason is that the same construction associated to different spaces reappears later in the proof and a more general treatment allows us to then refer to this section. We will associate a map $b: E_0^\xi \to \Sigma^{-|b|}E_0^\xi$ to any fibration
	$$ E_0 \to E_{p-1} \xrightarrow{q} J_{p-1}S^{2m} $$
	with a map $\xi: E_{p-1} \to BU$. When this construction is applied to the fibration
	$$ B_{p^j-1} \to B_{p^{j+1}-1} \xrightarrow{q} J_{p-1}S^{2p^jn}$$
	constructed in Lemma \ref{OGfibration}, it yields a map $b: G_j \to \Sigma^{-|b|}G_j$ since $G_j$ is the Thom spectrum associated to $B_{p^j-1}$. 
	
	\vspace{1em}
	
	Formally, consider the category with objects $(E_{p-1}, q, \xi)$ as above and in which the morphisms $f: (E_{p-1}, q, \xi) \to (E_{p-1}', q', \xi')$ are given by commutative diagrams
	$$
	\begin{tikzcd}
	E_{p-1} \arrow[rr, "f"] \arrow[dr, "\xi"] & & E_{p-1}' \arrow[dl, "\xi'"]\\
	&BU&
	\end{tikzcd}
	\quad \text{and} \quad
	\begin{tikzcd}
	E_{p-1} \arrow[rr, "f"] \arrow[dr, "q"] & & E_{p-1}' \arrow[dl, "q'"]\\
	&J_{p-1}S^{2m}&
	\end{tikzcd}.
	$$
	The association $F: (E_{p-1}, q, \xi) \mapsto E_0^\xi$ is a functor from this category to $\bf hSp$. Commutativity of the right triangle gives a map between the fibres and commutativity of the left triangle guarantees that this map is a morphism in ${\bf Top}_{BU}$. Passing to Thom spectra yields a morphism in $\bf hSp$. By construction it is now clear that $F$ is a functor. The map $b = b(E_{p-1}, q, \xi)$ we aim to construct turns out to be a natural transformation between $\Sigma^{2mp-2}F$ and $F$. 
	
	\vspace{1em}
	
	Let us begin by constructing $b$. By taking pullbacks of $q$ along the inclusion maps
	$$J_0S^{2m} \hookrightarrow J_1S^{2m} \hookrightarrow \dots \hookrightarrow J_{p-1}S^{2m}$$
	we obtain a commutative diagram
	$$
	\begin{tikzcd}
	E_0 \arrow[r, hook] \arrow[d]& E_1 \arrow[r, hook] \arrow[d]& \dots \arrow[r, hook] & E_{p-1} \arrow[d, "q"]\\
	J_0S^{2m} \arrow[r, hook] & J_1S^{2m} \arrow[r, hook]& \dots \arrow[r, hook]&J_{p-1}S^{2m}.
	\end{tikzcd}
	$$
	In particular $J_0S^{2m} = \{ * \} $ so $E_0$ is the fibre of $q$. 
	Consider the composition $E_{p-1} \xrightarrow{(q, \id)} J_{p-1}S^{2m} \times E_{p-1} \xrightarrow{\pi_2} E_{p-1} \xrightarrow{\xi}BU$. By passing to Thom spectra we obtain a map
	$$ E_{p-1}^\xi \to \Sigma^\infty J_{p-1}S^{2m}_+ \land E_{p-1}^\xi.$$
	Further define the maps
	$$ \theta_k : E_{p-1}^\xi \to \Sigma^\infty J_{p-1}S^{2m}_+ \land E_{p-1}^\xi \to S^{2mk} \land E_{p-1}^\xi \simeq \Sigma^{2mk} E_{p-1}^\xi$$
	where $ \Sigma^\infty J_{p-1}S^{2m}_+ \to S^{2mk}$ is obtained by choosing a splitting $ \Sigma^\infty J_{p-1}S^{2m}_+ \simeq \bigvee_{j=0}^{p-1}S^{2mj}$ due to the stable version of the Theorem \ref{James}. 
	
	\vspace{1em}
	
	There is a natural filtration $ E_0^\xi \to E_1^\xi \to \dots \to E_{p-1}^\xi $ of $E_{p-1}^\xi$ and we will be interested in the restrictions of $\theta_k$ to $E_{i}^\xi$ for $i \in \{ 1, \dots, p-1 \}$. At the moment, $E_i^\xi \to J_iS^{2m}_+ \land E_i^\xi$ is just a map of spectra, but we shall show that it actually preserves the smash product filtration on $J_iS^{2m}_+ \land E_i^\xi$, at least up to homotopy. We need this because it lets us establish some equivalences induced by the $\theta_k$ in the technical Lemma \ref{cofibre} and then we can finally use those to define $b$.
	
	\vspace{1em}
	
	The spheres are suspensions and hence co-$H$-spaces, so the diagonal map $S^{2m} \to S^{2m} \times S^{2m}$ and the co-$H$-space map $S^{2m} \to S^{2m} \vee S^{2m} \to S^{2m} \times S^{2m}$ are homotopic. By applying the partial James construction we obtain a homotopy $H: J_{p-1}S^{2m} \times I \to J_{p-1}S^{2m} \times J_{p-1}S^{2m}$ between the diagonal map $\Delta_0$ and the map $\Delta_1$ induced by the co-$H$-space map
	\begin{align*}
	\Delta_0: J_{p-1}S^{2m} &\to J_{p-1}S^{2m} \times J_{p-1}S^{2m}\\
	\Delta_1: J_{p-1}S^{2m} &\to J_{p-1}(S^{2m} \vee S^{2m}) \to J_{p-1}S^{2m} \times J_{p-1}S^{2m}.
	\end{align*}
	Note that if $J_{p-1}S^{2m}$ is filtered using the James filtration $J_0S^{2m} \hookrightarrow J_1S^{2m} \hookrightarrow \dots \hookrightarrow J_{p-1}S^{2m}$ and we equip $J_{p-1}S^{2m} \times J_{p-1}S^{2m}$ with the product filtration (in which the degree $k$ terms are $\bigcup_{i=0}^k J_iS^{2m} \times J_{k-i}S^{2m}$), then $\Delta_1$ is filtration preserving.
	
	\vspace{1em}
	
	In the next paragraph we try to lift $\Delta_0$, $\Delta_1$ and $H$.	Since $q$ is a fibration, so is $\id \times q: J_{p-1}S^{2m} \times E_{p-1} \to J_{p-1}S^{2m} \times J_{p-1}S^{2m}$. Using the homotopy lifting property we can lift $H$ to the diagonal of the diagram
	$$
	\begin{tikzcd}
	E_{p-1} \times I \arrow[r, "\widetilde{H}"] \arrow[d, "{q \times \id}"] & J_{p-1}S^{2m} \times E_{p-1} \arrow[d, "{\id  \times  q}"]\\
	J_{p-1}S^{2m} \times I \arrow[r, "H"] & J_{p-1}S^{2m} \times J_{p-1}S^{2m}.
	\end{tikzcd}
	$$
	and use this diagonal to define the top map $\widetilde{H}$ such that the square commutes. It is now easy to verify that $\widetilde{H_0} := \widetilde{H}|_{E_{p-1} \times \{0\}} = (q, \id)$ and hence by passing to Thom spectra we recover the map $E_{p-1}^\xi \to \Sigma^\infty J_{p-1}S^{2m}_+ \land E_{p-1}^\xi$ that appears in the definition of $\theta_k$. The homotopy $\widetilde{H}$ shows that it is homotopic to a map obtained by Thomifying $\widetilde{H_1} := \widetilde{H}|_{E_{p-1}\times \{1\}}$. But $\widetilde{H_1}$ is filtration preserving even before passing to Thom spectra because $\Delta_1$ is and the square commutes.
	
	
	\vspace{1 em}
	
	We are now well-prepared for the following lemma, which is the technical heart of this section. Studying its proof may be skipped during the first reading or altogether. The purpose of this lemma is to establish some equivalences that are required to give the definition of $b$. Notationally, we follow our convention that all restrictions of $\theta_k$ and all maps on cofibres induced by $\theta_k$ will also be denoted by $\theta_k$. This convention is useful, because looking at the expressions such as $ {\theta_k|}_{E_i^\xi / E_{i-1}^\xi}: \frac{E_i^\xi}{E_{i-1}^\xi} \to \Sigma^{2mk} \frac{E_{i-k}^\xi}{E_{i-k-1}^\xi}$ hurts.
	
	\begin{lemma}\label{cofibre}
		(i)The composition
		$ \frac{E_i^\xi}{E_{i-1}^\xi} \xrightarrow{\theta_k}\Sigma^{2mk} \frac{E_{i-k}^\xi}{E_{i-k-1}^\xi} \xrightarrow{\theta_j} \Sigma^{2m(k+j)} \frac{E_{i-k-j}^\xi}{E_{i-k-j-1}^\xi}$
		is equal to $\theta_{k+j}$ for $k+j \in \{1, \dots, p-1 \}$.\\
		(ii) There is an equivalence $\theta_k: E_{k}^\xi/E_{k-1}^\xi \to \Sigma^{2mk}E_0^\xi$ for any $k \in \{1, \dots, p-1 \}$.\\
		(iii) There is an equivalence $\theta_1: E_{p-1}^\xi / E_0^\xi \to \Sigma ^{2m} E_{p-2}^\xi$.
	\end{lemma}
	\begin{proof}
		We prove parts (ii) and (iii) and refer to \cite[Lemma 3.9]{devinatz1988nilpotence} for part (i).
		
		\vspace{1em}
				 
		(ii) We first establish a commutative diagram of pairs of spaces
		$$
		\begin{tikzcd}
		(* \times E_k, *\times E_{k-1}) \arrow[r] \arrow[d, "* \times q"] & ((S^{2mk} \times E_0) \cup (* \times E_k), * \times E_k) \arrow[d, "{\pi_1  \cup * \times q}"]\\
		(* \times J_kS^{2m}, * \times J_{k-1}S^{2m}) \arrow[r] & (S^{2mk} \vee J_kS^{2m}, * \times J_kS^{2m}).
		\end{tikzcd}
		$$
		It is clear that the vertical maps are well-defined maps of pairs. The bottom map $* \times J_k S^{2m} \to S^{2mk} \vee J_kS^{2m}$ is given by
		$$ * \times J_kS^{2m} \cong J_kS^{2m} \xrightarrow{\Delta_1} \bigcup_{i=0}^k J_iS^{2m} \times J_{k-i} S^{2m} \xrightarrow{\delta} S^{2mk} \times * \cup * \times \bigcup_{i=0}^{k-1} J_{k-i}S^{2m} $$
		where the first map is well-defined because $\Delta_1$ is filtration preserving. The second map $\delta$ is a projection onto the second factor for all $i \neq k$ and it is induced by $J_kS^{2m} \to J_kS^{2m} / J_{k-1}S^{2m} \cong {(S^{2m})}^{\land k} \cong S^{2mk}$ for $i=k$. Note that because $\Delta_1$ preserves the filtration on $J_kS^{2m}$ we have 
		$$\delta \left(\Delta_1 \left(* \times J_{k-1}S^{2m} \right) \right) \subset \delta \left(\bigcup_{i=0}^{k-1}J_iS^{2m} \times J_{k-1-i}S^{2m}\right) \subset * \times J_kS^{2m}.$$
		Therefore the bottom map is a map of pairs. Similarly, the top map can be defined using $\widetilde{H_1}$ in place of $\Delta_1$ and it similarly follows that it is a map of pairs. Note that we are crucially using the fact that $\Delta_1$ and $\widetilde{H_1}$ are filtration preserving.
		
		\vspace{1em}
		
		The map $H_*(*\times J_kS^{2m}, *\times J_{k-1}S^{2m};\F_p) \to H_*(S^{2mk} \vee J_kS^{2m}, * \times J_kS^{2m};\F_p)$ is an isomorphism because it is induced by the map
		$$ \frac{J_{k}S^{2m}}{J_{k-1}S^{2m}} \cong S^{2mk} \xrightarrow{\id} S^{2mk} \cong \frac{S^{2mk} \vee J_kS^{2m}}{* \times J_kS^{2m}}$$
		which can be seen to be an identity by decompressing the definition of the bottom map of pairs. The diagram is a pullback square of pairs by an explicit calculation and hence the top map is an equivalence on homology too. By passing to the relative Thom spectra from the top map of pairs we obtain that
		$$ \frac{E_k^\xi}{E_{k-1}^\xi} \to \Sigma^{2mk}E_0^\xi $$
		is an equivalence by the homology Whitehead theorem and the Thom isomorphism theorem.
		
		\vspace{1em}
		
		(iii) We show by induction on $k$ that $E_k^\xi / E_0^\xi \xrightarrow{\theta_1} \Sigma^{2m} E_{k-1}^\xi$ is an equivalence for all $k \in \{ 1, \dots, p-1 \}$. The base case $k=1$ was shown in part (ii).
		
		\vspace{1em}
		
		Consider the following diagram in which both rows are cofibre sequences and the vertical maps are induced by $\theta_1$.
		$$
		\begin{tikzcd}
		\frac{E_{k-1}^\xi}{E_0^\xi} \arrow[r] \arrow[d, "\theta_1"]& \frac{E_{k}^\xi}{E_0^\xi} \arrow[r] \arrow[d, "\theta_1"]& \frac{E_{k}^\xi}{E_{k-1}^\xi} \arrow[d, "\theta_1"]\\
		\Sigma^{2m}E_{k-2}^\xi \arrow[r]& \Sigma^{2m}E_{k-1}^\xi \arrow[r]& \Sigma^{2m} \frac{E_{k-1}^\xi}{E_{k-2}^\xi}
		\end{tikzcd}
		$$
		To distinguish between the maps induced by $\theta_1$ we call them the leftmost, the middle and the rightmost map. The leftmost map is an equivalence by the induction hypothesis. The rightmost map can be composed with $\theta_{k-1}$ to form
		$$ \frac{E_k^\xi}{E_{k-1}^\xi} \xrightarrow{\theta_1}\Sigma^{2m} \frac{E_{k-1}^\xi}{E_{k-2}^\xi} \xrightarrow{\theta_{k-1}} \Sigma^{2mk} E_0^\xi$$
		which is $\theta_k$ by part (i). But $\theta_k$ is an equivalence by (ii). It follows that $\theta_{k-1} \circ \theta_1$ is an equivalence and so the rightmost map $\theta_1$ is an equivalence too.
		
		\vspace{1 em}
		
		Consider the commutative diagram of long exact sequences on (mod $p$) homology induced by the cofibre sequences above. The left and the right map are isomorphisms and so by the 5-lemma, the middle map is an isomorphism on too. By the homology Whitehead theorem, the middle map $\theta_1$ is an equivalence which completes the inductive step.
	\end{proof}
	Having established Lemma \ref{cofibre}, we can now finally define $b: \Sigma^{2mp-2}E_0^\xi \to E_0^\xi$ as the composite
	$$  \Sigma^{2mp-2} E_0^\xi \xrightarrow{\theta_{p-1}^{-1}} \Sigma^{2m-2} \frac{E_{p-1}^\xi}{E_{p-2}^\xi} \xrightarrow{\delta_1} \Sigma^{2m-1}E_{p-2}^\xi \xrightarrow{\theta_1^{-1}} \Sigma^{-1} \frac{E_{p-1}^\xi}{E_0^\xi} \xrightarrow{\delta_2} E_0^\xi$$ 
	where the maps $\delta_1$ and $\delta_2$ arise from the cofibre sequences
	$$ \dots \to \Sigma^{2m-2} E_{p-2}^\xi \to \Sigma^{2m-2} E_{p-1}^\xi \to \Sigma^{2m-2} \frac{E_{p-1}^\xi}{E_{p-2}^\xi} \xrightarrow{\delta_1} \Sigma^{2m-1} E_{p-2}^\xi \to \dots$$
	and
	$$ \dots \to \Sigma^{-1} E_{0}^\xi \to \Sigma^{-1} E_{p-1}^\xi \to \Sigma^{-1} \frac{E_{p-1}^\xi}{ E_{0}^\xi} \xrightarrow{\delta_2} E_{0}^\xi \to \cdots.$$
	Most concisely the map $b$ can be expressed as $\delta_2 \circ \theta_1^{-1}\circ \delta_1 \circ \theta_{p-1}^{-1}$. The useful features of this formula are the facts that $\theta_1$ and $\theta_{p-1}$ are equivalences and the maps $\delta_1$ and $\delta_2$ both have some suspension of $E_{p-1}^\xi$ as a cofibre.
	
	\vspace{1em}
	
	\begin{remark}
		We reiterate that all spaces and spectra in the discussion are $p$-local. The reason for this assumption is not apparent from the proof presented here since the $p$-locality is only used in the omitted argument for part (i). Looking at this step carefully, one establishes that the statement of the Lemma \ref{cofibre}(i) can be slightly generalized to only requiring the invertibility of certain integers (as opposed to localization at $p$). Nonetheless, some integers will always need to be invertible which hints at the inherently local structure of the proof of the nilpotence theorem.
		
		
		
				
	\end{remark}

	\vspace{1em}
	
	For the proof that $b$ is a natural transformation see \cite[Proposition 3.15]{devinatz1988nilpotence}. 
	\begin{lemma}
		The map $b=b(E_{p-1}, q, \xi): E_0^\xi \to \Sigma^{-|b|}E_0^\xi$ is a natural transformation.
	\end{lemma}
	
	\subsubsection{Bousfield equivalence classes of $G_j$ and $G_{j+1}$}
	In the previous section, we have associated a map $b: E_0^\xi \to \Sigma^{-|b|}E_0^\xi$ to a general framework consisting of any fibration  $E_0 \to E_{p-1} \xrightarrow{q} J_{p-1}S^{2m}$ with a map $\xi: E_{p-1} \to BU$. While we tried to motivate the lengthy and difficult construction by promises that $b$ is an essential ingredient of the proof, we never provided any evidence for our claims. In this section, we deliver on our promises by using $b$ to relate the Bousfield equivalence classes of $G_j$ and $G_{j+1}$.
	
	\vspace{1em}
	
	Let us specialize to the fibration $B_{p^j-1} \to B_{p^{j+1}-1} \xrightarrow{q} J_{p-1}S^{2p^jn}$ obtained in the Lemma \ref{OGfibration}, but we still use $m:=p^jn$ as a shorthand in many expressions. In this instance we have $E_0^\xi= G_j$, $E_{p-1}^\xi = G_{j+1}$ and the map $b:G_j \to \Sigma^{-|b|}G_j$. In the argument, both of these notations are used interchangeably because they are useful for different reasons. Globally, we are only interested in $G_j$ and $G_{j+1}$, but their relationship arises through them being the endpoints of some filtration with specific properties.
	
	\vspace{1em}
	
	Our ultimate goal is to establish the Bousfield equivalence $\langle G_j \rangle  = \langle G_{j+1} \rangle$. The following lemma provides more than one of the two inequalities.
	
	\begin{lemma}
		$\langle G_j \rangle  = \langle G_{j+1} \rangle \vee \langle b^{-1}G_j \rangle$.
	\end{lemma}	
	\begin{proof}
		Let $X$ be any spectrum such that $X \land G_j$ is contractible. We need to show that $X \land (G_{j+1} \vee b^{-1}G_j)= (X \land G_{j+1}) \vee (X \land b^{-1}G_j)$ is contractible and we do this by proving that both $X \land G_{j+1}$ and $X \land b^{-1}G_j$ are contractible.
		
		\vspace{1em}
		
		The spectrum $X \land b^{-1}G_j$ is the homotopy colimit of the diagram
		$$  X \land G_j \xrightarrow{\id_X \land b} X \land \Sigma^{-|b|}G_j \xrightarrow{\id_X \land b} X \land \Sigma^{-2|b|} G_j \to \cdots $$
		because the smash product commutes with arbitrary homotopy colimits. Since $X \land G_j$ is contractible, all spectra $X \land \Sigma^{-m|b|}G_j$ are contractible and hence so is their colimit $X \land b^{-1}G_j$.
		
		\vspace{1em}
		
		To show that $X \land G_{j+1}$ is contractible, we use the spectra $X \land E_k^\xi$ to interpolate between $X \land G_j$ and $X \land G_{j+1}$.
		For any  $k \in \{ 1, \dots, p-1 \}$ there is a cofibration $ E_{k-1}^\xi \to E_k^\xi \to E_k^\xi/E_{k-1}^\xi$ and recall that $\theta_k: E_k^\xi/E_{k-1}^\xi \to \Sigma^{2mk}G_j$ is an equivalence by Lemma \ref{cofibre} (ii). Smashing by $X$ we obtain a cofibration
		$$ X \land E_{k-1}^\xi \to X \land E_k^\xi \to X \land E_k^\xi/E_{k-1}^\xi \simeq X \land \Sigma^{2mk}G_j \simeq *$$
		in which the cofibre is contractible because it is a suspension of $X \land G_j$. By passing to the long exact sequence of homotopy groups associated to this cofibre sequence we find that $\pi_d(X \land E_k^\xi) \cong \pi_d(X \land E_{k-1}^\xi)$ for all $k$ and $d$. Using induction on $k$ one can show that $\pi_d(X \land E_{p-1}^\xi) \cong \pi_d(X \land E_0^\xi)$ for all $d$. But $X \land G_j$ is contractible by assumption and hence so is $X \land G_{j+1}$.
		
		\vspace{1em}
		
		We now prove the converse. Let $X$ be any spectrum such that $X \land G_{j+1}$ and $X \land b^{-1}G_j$ are both contractible. The definition of $b$ is $b = \delta_2 \circ \theta_1^{-1}\circ \delta_1 \circ \theta_{p-1}^{-1}$ where the cofibres of $\delta_1$ and $\delta_2$ are given by
		$$  C_{\delta_1}= \Sigma^{-1+2m}G_{j+1} \quad \text{and} \quad C_{\delta_2}=G_{j+1}.$$
		In particular, both $C_{\delta_1}$ and $C_{\delta_2}$ are suspensions of $G_{j+1}$ and hence contractible upon smashing with $X$. By considering the cofibre sequence of $\id_X \land \delta_2$
		$$ * \simeq X \land \Sigma^{-1}G_{j+1} \to X \land \Sigma^{-1} \frac{G_{j+1}}{G_j} \xrightarrow{\id_X \land \delta_2} X \land G_j \to  X \land G_{j+1} \simeq *$$
		it follows that $\id_X \land \delta_2$ is an equivalence. By an analogous argument the map $\id_X \land \delta_1$ is an equivalence. In Lemma \ref{cofibre} we have shown that the maps $\theta_1$ and $\theta_{p-1}$ are equivalences too and hence so are $\id_X \land \theta_1^{-1}$ and $\id_X \land \theta_{p-1}^{-1}$. 
		
		\vspace{1em}
		
		Consider now the map $\id_X \land b: X \land G_j \to X \land \Sigma^{-|b|}G_j$ which is an equivalence because it is a composition of four equivalences. Therefore all maps in the diagram
		$$ X \land G_j \xrightarrow{\id_X \land b} X \land  \Sigma^{-|b|} G_j \xrightarrow{\id_X \land b} \cdots  $$
		are equivalences and hence so is the map $X \land G_j \to X \land b^{-1}G_j$. Using the assumption that $b^{-1}G_j$ is contractible we obtain that $X \land G_j$ is contractible as required.
	\end{proof}

	\subsubsection{$b^{-1}G_j$ is contractible}
	The map $b$ was used to relate the Bousfield equivalence classes $\langle G_j \rangle$ and $\langle G_{j+1} \rangle$ in the previous section. To show that they are equal, it remains to establish that the telescope $b^{-1}G_j$ is contractible.
	
	However, this is impossible to do directly with what we know about $b$ so far. In this section, we first establish a factorisation of $b$ through $\Omega^2 S^{2m+1}_+ \land G_j$ given by the commutative diagram
	$$
	\begin{tikzcd}
	\Sigma^{2mp-2}G_j \arrow[d, equal, "\id"] \arrow[r, "b"] & G_j \\
	S^{2mp-2} \land G_j \arrow[r, "\gamma \land \id"]& \Sigma^\infty \Omega^2 S^{2m+1}_+ \land G_j \arrow[u, "\mu"]
	\end{tikzcd}
	$$
	where $\gamma$ and $\mu$ are some maps yet to be defined. This allows us to utilize the Snaith's splitting structure of $\Sigma^\infty \Omega^2 S^{2m+1}_+$ as follows.
	
	Let $ \Sigma^\infty \Omega^2 S^{2m+1}_+ \simeq \bigvee_{k=0}^\infty D_k$ be the Snaith's splitting. A careful analysis of $\gamma$ coupled with some properties of the splitting shows that $\gamma: S^{2mp-2} \to \Sigma^\infty \Omega^2 S^{2m+1}_+$ factors through $D_p$. It is then of interest to study the homotopy colimit of the diagram
	$$ S \xrightarrow{\gamma} \Sigma^{-|b|}D_p \xrightarrow{\gamma} \Sigma^{-2|b|} D_{2p} \xrightarrow{\gamma} \cdots$$
	which evaluates to $H \F_p$ by a result due to Mahowald \cite{mahowald1977new}. This shows, after passing from our factorization diagram to homotopy colimits in the appropriate sense, that the identity map $\id_{b^{-1} G_j}$ factors through $H\F_p \land G_j$. To conclude, we prove that $H \F_p \land G_j$ is contractible. Hence so is $b^{-1}G_j$.
	
	\vspace{1 em}
	
	We have been very hand-wavy and conceptual in our summary of the argument. Let us now sink into rigour and technicalities.
	
	\paragraph{\bf The map $\mu$} Let us begin by constructing a map $\mu:\Sigma^\infty \Omega^2 S^{2m+1}_+ \land G_j \to G_j$. We first define the the map in ${\bf Top}_{BU}$ and then stabilize by passing to Thom spectra. Recall that any map $f:X \to Y$ in $\bf Top$ can be replaced with a fibration $p_f: I_f \to Y$ up to homotopy equivalence, \emph{i.e.} such that $X \simeq I_f$. We apply this construction to the fibration $B_{p^j-1} \to \Omega SU(n+1) \xrightarrow{h} \Omega S^{2m+1}$ established in the proof of Lemma \ref{OGfibration}. For any $* \in \Omega S^{2m+1}$ we obtain
	\begin{align*}
		I_{h} &= \{  (\omega, e) \in P \Omega S^{2m+1} \times \Omega SU(n+1) \ | \ \omega(0)=h(e) \}\\
		p_{h}^{-1}(*) &= \{  (\omega, e) \in P \Omega S^{2m+1} \times \Omega SU(n+1) \ | \ \omega(0)=h(e), \ \omega(1) = * \}
	\end{align*}
	together with homotopy equivalences $\Omega SU(n+1) \simeq I_{h}$ and $B_{p^j-1} \simeq p_{h}^{-1}(*)$. Define now
	\begin{align*}
		\mu: P_*\Omega S^{2m+1} \times p_{h}^{-1}(*) &\to I_h\\
		(\lambda, (\omega, e)) &\mapsto (\lambda \omega, e)
	\end{align*}
	and note that it restricts to $\mu: \Omega^2 S^{2m+1} \times  p_{h}^{-1}(*) \to  p_{h}^{-1}(*)$. Passing to Thom spectra now yields the map $\mu: \Sigma^\infty \Omega^2 S^{2m+1}_+ \land G_j \to G_j$ as required.

	
	\vspace{1em}
	
	\begin{remark}
		At this point our presentation of the proof differs from that of \cite{devinatz1988nilpotence} in more than just the order in which the material is presented; the definition of the action $\mu$ is different. Indeed, this is only a superficial disparity. Having previously established the fibration $\Omega SU(n+1) \to \Omega^2 S^{2m+1}$ allows us to immediately define the action $\Sigma^\infty \Omega^2S^{2m+1}_+ \land G_j \xrightarrow{\mu} G_j$ rather than defining the action $\Sigma^\infty \Omega J_{p-1} S^{2m}_+ \land G_j \to G_j$ and extending it later (which \cite{devinatz1988nilpotence} does implicitly through a commutative diagram of fibrations established in Lemma 3.27 of that paper). Extending the action to $\Sigma^\infty \Omega^2 S^{2m+1}_+$ at some point is crucial because it allows us to utilize the nice Snaith's splitting structure of that spectrum.
	\end{remark}
	
	\paragraph{\bf The map $\gamma$} We now define the map of spectra $\gamma:S^{2mp-2} \to \Sigma^\infty \Omega^2 S^{2m+1}_+$. When we constructed $b$ at the start of this section, we resorted to a general framework in order to reuse the same construction at a later stage in the proof. This time has now arrived.
	
	Let $* \in J_{p-1}S^{2p^jn}$ be the basepoint and let $s$ denote the path fibration
	$\Omega J_{p-1}S^{2p^jn} \to P_* J_{p-1}S^{2p^jn} \xrightarrow{s} J_{p-1}S^{2p^jn}.$
	Defining $0:P_*J_{p-1}S^{2p^jn} \to BU$ to be the constant zero map places us in the situation we have already encountered: there is a map $$b''=b(P_* J_{p-1}S^{2p^jn}, s,0): \Sigma^{2mp-2} \Sigma^\infty \Omega J_{p-1}S^{2p^jn}_+ \to  \Sigma^\infty \Omega J_{p-1}S^{2p^jn}_+$$ associated to this setting. Define
	$$ \gamma: \Sigma^{2mp-2}S  \hookrightarrow \Sigma^{2mp-2} \Sigma^\infty \Omega J_{p-1}S^{2p^jn}_+ \xrightarrow{b''} \Sigma^\infty \Omega J_{p-1}S^{2p^jn}_+ \hookrightarrow\Sigma^\infty \Omega^2 S^{2p^jn+1}_+$$
	where the left map is just the inclusion of the bottom cell.
	
	
	\vspace{1em}
	
	The reason for introducing the maps $\mu$ and $\gamma$ is that they can be composed to form an interesting map 	
	$$ \Sigma^{2mp-2} G_j = S^{2mp-2} \land G_j \xrightarrow{\gamma \land \id} \Sigma^\infty \Omega^2 S^{2m+1}_+ \land G_j \xrightarrow{\mu} G_j.$$
	This turns out to be another name for $b$.
	\begin{lemma}\label{factorizationOfb}
		The diagram
		$$
		\begin{tikzcd}
		\Sigma^{2mp-2}G_j \arrow[d, equal, "\id"] \arrow[r, "b"] & G_j \\
		S^{2mp-2} \land G_j \arrow[r, "\gamma \land \id"]& \Sigma^\infty \Omega^2 S^{2m+1}_+ \land G_j \arrow[u, "\mu"]
		\end{tikzcd}
		$$
		commutes.
	\end{lemma}
	\begin{proof}
	Consider the map of fibrations
	$$
	\begin{tikzcd}
	\Omega J_{p-1} S^{2p^jn} \times p_q^{-1}(*) \arrow[r] \arrow[d, "\mu"]& P_*J_{p-1} S^{2p^jn} \times p_q^{-1}(*) \arrow[r, "s \pi_1"] \arrow[d, "\mu"]& J_{p-1} S^{2p^jn} \arrow[d, equal, "\id"]\\
	p_q^{-1}(*) \arrow[r]& I_q \arrow[r, "p_q"]& \Omega S^{2p^jn+1}
	\end{tikzcd}
	$$
	where we recall from the discussion preceding this proof that the bottom row is a fibration replacement of $q$ (in particular $B_{p^{j+1}-1} \simeq I_q$ and $B_{p^j-1} \simeq p_q^{-1}(*)$) and the top row is obtained from the path fibration.
	
	
	\vspace{1em}

	The original map $b = b(B_{p^{j+1}-1}, q, \xi): \Sigma^{|b|}G_j \to G_j$ is the map associated to the bottom fibration. The map $b'=b(P_* J_{p-1}S^{2p^jn} \times B_{p^j-1}, s  \pi_1, 0)$ is the map associated to the top fibration. By passing to Thom spectra and using the naturality of the construction $b(E_{p-1}, q, \xi)$ we obtain a commutative diagram
	$$
	\begin{tikzcd}
	\Sigma^{|b|} \Sigma^\infty \Omega J_{p-1}S^{2p^jn}_+ \land G_j \arrow[r, "b'"] \arrow[d, "\mu"]&\Sigma^\infty \Omega J_{p-1}S^{2p^jn}_+ \land G_j \arrow[d, "\mu"] \\
	\Sigma^{|b|}G_j \arrow[r, "b"] & G_j.
	\end{tikzcd}
	$$
	Let us now rewrite $b'$ in a nicer form. We have
	$$ b(P_* J_{p-1}S^{2p^jn} \times B_{p^j-1}, s \circ \pi_1, 0) = b(P_* J_{p-1}S^{2p^jn}, s, 0) \land \id_{G_j} $$
	by Lemma 3.20 in \cite{devinatz1988nilpotence}, which is not hard to prove. By definition $b'' = b(P_*\Omega S^{2p^jn+1}, s, 0)$ so the equality can be expressed succinctly as $b' = b'' \land \id_{G_j}$. The commutative diagram from above yields the diagram
	
	$$
	\begin{tikzcd}
	\Sigma^{|b|} G_j \arrow[r] &\Sigma^{|b|}\Sigma^{\infty} \Omega^2 S^{2p^jn+1}_+ \land G_j \arrow[r, "b'' \land \id"] \arrow[d, "\mu"]& \Sigma^{\infty} \Omega^2 S^{2p^jn+1}_+ \land G_j \arrow[d, "\mu"] \\
	&\Sigma^{|b|}G_j \arrow[r, "b"] & G_j
	\end{tikzcd}
	$$
	in which the top row is the map $\gamma \land \id$. The bottom composition is the map $b$. Therefore $b= \mu \circ (\gamma \land \id)$ as required.
	\end{proof} 
		
	As we have remarked earlier, the reason for immense usefulness of this factorization of $b$ lies in the very nice splitting structure of the spectrum $\Sigma^{\infty}\Omega^2 S^{2m+1}_+$. This result is known under the name of Snaith's splitting.
	
	\begin{theorem}[Snaith's splitting of $S^{2m-1}_+$]There is a decomposition
		$$ \Sigma^{\infty} \Omega^2 S^{2m+1}_+ \simeq \bigvee_{k=0}^\infty D_k$$
		where the $D_k$ are finite spectra.
	\end{theorem}
	This splitting has several convenient properties we will use without proof in the discussion below. A reader is referred to \cite{cohen1981odd} for a comprehensive treatment of the spaces of the form $\Omega^2 S^k$ and their suspension spectra.
	
	\vspace{1em}
	
	We are interested to see how $\gamma:S^{2mp-2} \to \Sigma^{\infty} \Omega^2 S^{2m+1}_+$ interacts with this splitting structure. It turns out that up to $H\F_p$-equivalence this interaction is as nice as possible.
	
	\begin{lemma}\label{Dpfactorization}
		The map $\gamma:S^{2mp-2} \to \Sigma^\infty \Omega^2 S^{2m+1}_+$ factors through $D_p$.
	\end{lemma}
	This will allow us to rewrite $b$ in a much more powerful form and establish the contractibility of $b^{-1}G_j$.
	
	\begin{proof}	
	We first investigate the relationship between the (mod $p$) homologies of $\Sigma^{\infty} \Omega^2 S^{2m+1}_+$ and those of the spectra $D_k$ on the other side of the Snaith's splitting. It can be shown that the (mod $p$) homology of $\Omega^2 S^{2m+1}_+$ is
	$$H_*(\Omega^2 S^{2m+1}_+ ; \F_p) \cong \Lambda \left[ x_{2m-1}, x_{2mp-1}, \dots \right] \otimes \F_p \left[y_{2mp-2} ,y_{2mp^2-2}, \dots \right]$$
	where the subscripts $i$ indicate the homological degree of the classes $x_i$ and $y_i$ and $\Lambda$ denotes the exterior algebra on the given generators. $\Sigma^{\infty} \Omega^2 S^{2m+1}_+$ is a suspension spectrum so it has the same homology.
	
	\vspace{1em}
	
	The vector space $H_*(\Omega^2 S^{2m+1}_+; \F_p)$ is naturally graded in which the grading is given by the homological degree. To describe the homology of the spectra $D_k$ as a subspace, we introduce another grading by giving \emph{weight}
	$$ \text{wt}(x_{2mp^i-1}) = \text{wt}(y_{2mp^i-2}) = p^i$$
	to the generators and extending using $\text{wt}(ab)=\text{wt}(a) + \text{wt}(b)$. With this notation $H_*(D_k; \F_p) \subset H_*(\Omega^2 S^{2m+1}_+; \F_p)$ is the vector space spanned by all monomials of weight $k$.
	
	\vspace{1 em}
	
	In particular, observe that $\text{wt}(x_{2m-1})=1$ and every other generator of $H_*(\Omega^2S^{2m+1}_+ ; \F_p)$ has weight divisible by $p$. Since $x_{2m-1}$ is an element of the exterior algebra we have $x_{2m-1}^2=0$. Hence all monomials in $H_*(\Omega^2 S^{2m+1}_+ ; \F_p)$ have weight $0$ or $1 \pmod p$ and $H_*(D_k; \F_p) =0$ unless $k \equiv 0$ or $1 \pmod p$.
	
	\vspace{1em}
	
	To see that $\gamma$ factors through $D_p$ consider the map
	$$\Sigma^{\infty} \Omega^2 S^{2m+1}_+ \xrightarrow{\simeq} \bigvee_{k=0}^\infty D_k \xrightarrow{\varepsilon_+ \vee g} D_0 \vee D_1 \vee D_p = S^{2m-1}_+ \vee D_p$$
	using the fact that $D_0=S^0$ and $D_1=S^{2m-1}$. By homology considerations above, this is a $H \F_p$-equivalence in dimensions up to $(2m-1)+(2mp-2)=2mp+2m-3$. It is another known result about the Snaith's splitting that the map $\varepsilon_+$ is obtained from the evaluation map of spaces 
	\begin{align*}
	S^1 \land S^1 \land \Omega^2 S^{2m+1} &\to S^{2m+1}\\
	(x, y, f) &\mapsto f(y)(x)
	\end{align*}
	after passing to $\bf hSp$.
	Precomposing with $\gamma$ we obtain
	$$S^{2mp-2} \hookrightarrow \Sigma^{2mp-2} \Sigma^{\infty} \Omega^2 S^{2m+1}_+ \xrightarrow{b''} \Sigma^{\infty} \Omega^2 S^{2m+1}_+ \xrightarrow{\varepsilon_+ \vee g} S^{2m-1}_+ \vee D_p.$$
	Now it turns out that $\varepsilon_+ \circ b''$ is null-homotopic. This can be seen by expanding the definition of $b''$ as the composition $\delta_2 \circ \theta_1^{-1}\circ \delta_1 \circ \theta_{p-1}^{-1}$ where these maps are associated to the triple $(P_*\Omega S^{2m+1}, s, 0)$ and studying $\theta_1$. See \cite[Corollary 3.26]{devinatz1988nilpotence} for details.
	
	Since $\varepsilon_+ \circ b''$ is null-homotopic, it follows that $\gamma$ is homotopic to a map into $D_p$ as required.
	\end{proof}
	Everything we do in the rest of the proof only concerns the homotopy class of $\gamma$, so we without loss of generality assume that $\gamma: S^{2mp-2} \to D_p$ lands in $D_p$.
	
	\vspace{1em}
	
	Suspensions of the map $\gamma: S^{2mp-2} \to D_p$ can be used to create a telescope
	$$ S \xrightarrow{\gamma} \Sigma^{-|b|}D_p \xrightarrow{\gamma} \Sigma^{-2|b|} D_{2p} \xrightarrow{\gamma} \cdots.$$
	Its homotopy colimit is $\varinjlim \Sigma^{-N|b|}D_{Np} = H\F_p$
	the (mod $p$) Eilenberg-Mac Lane spectrum due to Mahowald \cite{mahowald1977new}.
	
	\vspace{1em}
	
	Recall that the map $b$ factors as
	$ S^{2mp-2} \land G_j \xrightarrow{\gamma \land \id } \Omega^2 S^{2m+1}_+ \land G_j \xrightarrow{\mu} G_j$
	by Lemma \ref{factorizationOfb}. We have just shown that $\gamma$ lands in $D_p$ and therefore $b$ can be rewritten as
	$$ S^{2mp-2} \land G_j \xrightarrow{\gamma \land \id} D_p \land G_j \hookrightarrow \Sigma^{\infty} \Omega^2 S^{2m+1}_+ \land G_j \xrightarrow{\mu} G_j. $$
	It follows that for $N \in \N$ we can write $b^N$ as
	$$ S^{N(2mp-2)} \land G_j \xrightarrow{\gamma^{\land N} \land \id} D_p^{\land N} \land G_j \hookrightarrow (\Sigma^{\infty} \Omega^2 S^{2m+1}_+)^{\land N} \land G_j \xrightarrow{\mu^{\land N}} G_j.$$
	We now use another property of the Snaith's splitting. The loop space $\Omega^2S^{2m+1}$ is an $H$-space. Passing to $\bf hSp$ the $H$-space structure induces a multiplication map $ \Sigma^{\infty} \Omega^2S^{2m+1}_+ \land \Sigma^{\infty} \Omega^2S^{2m+1}_+ \to \Sigma^{\infty} \Omega^2S^{2m+1}_+ $ compatible with the Snaith's splitting. In particular, the multiplication induces the maps $D_{ip} \land D_{jp} \to D_{(i+j)p}$ such that the diagrams
	$$
	\begin{tikzcd}
	S^{i|b|} \land S^{j|b|} \arrow[d, "\gamma^i \land \gamma^j"] \arrow[r, equal, "\id"]& S^{(i+j)|b|} \arrow[d, "\gamma^{i+j}"]\\
	D_{ip} \land D_{jp} \arrow[r]& D_{(i+j)p}	
	\end{tikzcd}
	$$
	commute. Therefore $b^N$ simplifies to
	$$ S^{N(2mp-2)} \land G_j \xrightarrow{\gamma^N \land \id} D_{Np} \land G_j \hookrightarrow \Sigma^{\infty} \Omega^2 S^{2m+1}_+ \land G_j \xrightarrow{\mu} G_j.$$ 
	This can be restated by saying that the diagram
	$$\begin{tikzcd}
	G_j \arrow[r, equal, "\id"] \arrow[d, "\gamma \land \id"]& G_j \arrow[d, "b"] \\
	\Sigma^{-|b|}D_p \land G_j \arrow[r, "\mu"] \arrow[d, "\gamma \land \id"]& \Sigma^{-|b|}G_j \arrow[d, "b"] \\
	\Sigma^{-2|b|}D_{2p} \land G_j \arrow[r, "\mu"] \arrow[d, "\gamma \land \id"]& \Sigma^{-2|b|}G_j \arrow[d, "b"] \\
	\vdots&\vdots
	\end{tikzcd}$$
	commmutes. By passing to the homotopy colimits of both columns we see that the map $G_j \to b^{-1}G_j$ factors through $\varinjlim \Sigma^{-N|b|}D_{Np} \land G_j = H\F_p \land G_j$. Hence we can consider an enlarged commutative diagram
	$$\begin{tikzcd}
	G_j \arrow[d, "b"] \arrow[r]& G_j \land H\F_p \arrow[r] \arrow[d, "b \land \id"] & b^{-1}G_j \arrow[d, equal, "\id"] \\
	\Sigma^{-|b|}G_j \arrow[d, "b"] \arrow[r]& \Sigma^{-|b|}G_j \land H\F_p \arrow[r] \arrow[d, "b \land \id"] & b^{-1}G_j \arrow[d, equal, "\id"] \\
	\vdots&\vdots&\vdots
	\end{tikzcd}$$
	where the vertical maps are induced by $b$. This allows us to pass to the homotopy colimits of the three columns. As always, using that the smash procuct commutes with colimits, we obtain a map
	$$  b^{-1}G_j \to b^{-1}G_j \land H\F_p  \to b^{-1}G_j.  $$
	This is the identity map on $b^{-1}G_j$ which can be seen by considering only the first and last columns of the commutative diagram.
	
	\vspace{1em}
	
	To say something about $ b^{-1} G_j \land H\F_p$ in the middle, let us study the map that $b$ induced on (mod $p$) homology. Recall that $b=\delta_2 \circ \theta_1^{-1} \circ \delta_1 \circ \theta_{p-1}^{-1}$ where $\theta_1$ and $\theta_{p-1}$ are equivalences. The map $\delta_2$ arises from the cofibre sequence
	$$ \dots \to \Sigma^{-1} G_j \to \Sigma^{-1} G_{j+1} \to \Sigma^{-1} G_{j+1} / G_j \xrightarrow{\delta_2} G_j \to \cdots.$$
	Recall from \ref{homologyofBk} that $H_*(B_{p^j-1};\F_p)$ is the $\F_p\left[x_1, \dots, x_{n-1} \right]$-submodule of $\F_p\left[x_1, \dots, x_n \right]$ generated by $1, x_n, \dots, x_n^{p^j-1}$, the inclusion $B_{p^j-1} \hookrightarrow B_{p^{j+1}-1}$ is a monomorphism on (mod $p$) homology. By exactness and the Thom isomorphism theorem $\delta_2$ is the zero map on (mod $p$) homology. Thus ${H\F_p}_*(b)=0$ too and $b^{-1}G_j \land H\F_p$ is contractible. This gives the factorization 
	$$b^{-1}G_j \to * \to b^{-1}G_j$$
	of the identity map on $b^{-1}G_j$. It follows that $b^{-1}G_j$ is contractible, which is exactly what we wanted to prove.
	
	\subsection{Conclusion}\label{conclusion}
	In Step I and Step II we have proven two major propositions.
	\begin{proposition}\label{stepI}
		Let $X(n+1)_*(\alpha)$ be nilpotent. Then $G_j \land \alpha^{-1}R$ is contractible for sufficiently large $j$.
	\end{proposition}
	\begin{proposition}\label{stepII}
		$\langle G_j \rangle = \langle G_{j+1} \rangle$ for any $j$.
	\end{proposition}
	The proof of the weak ring spectrum form of the nilpotence theorem now follows relatively easily. Consider the diagram
	$$
	\begin{tikzcd}
	S^d \arrow[rr, "\eta"] \arrow[dd, "\alpha"] \arrow[ddrr, "h(\alpha)"]& &MU \land S^d \arrow[dd, "\id_{MU} \land \alpha"]\\
	\\
	R \arrow[rr, "\eta \land \id_{R}"] & &MU \land R
	\end{tikzcd}
	$$
	defining the $MU$ Hurewicz's homomorphism $h$. Since $\alpha \in \ker h$, we have that $h(\alpha)$ is null-homotopic and so $MU_*(\alpha)=0$. Since $MU$ is the homotopy colimit $MU = \varinjlim X(n)$ of the spectra $X(n)$ by Lemma \ref{XnMU} we must have $X(n+1)_*(\alpha)=0$ for sufficiently large $n$. By Proposition \ref{stepI} it follows that $G_j \land \alpha^{-1}R$ is contractible for sufficiently large $j$. Using the Bousfield equivalence in the Proposition \ref{stepII} inductively, we can descend from $G_j$ to $G_0$ and thus show that $G_0 \land \alpha^{-1}R = X(n) \land \alpha^{-1}R$ is contractible as well. Smash products commute with homotopy colimits so $X(n) \land \alpha^{-1}R$ is the homotopy colimit of the diagram
	$$  X(n) \land R \xrightarrow{\id_{X(n)} \land \alpha} X(n) \land \Sigma^{-d}R \xrightarrow{\id_{X(n)} \land \alpha} \cdots.$$
	We take homotopy groups. Because $\pi_*(X(n)\land \alpha^{-1}R)=0$ by contractibility, it follows that $X(n)_*(\alpha)$ is nilpotent.
	
	\vspace{1em}
	
	We can iterate the above procedure $n$ times to obtain that $X(1)_*(\alpha)$ is nilpotent. But $X(1)=S$ so $X(1)_*(\alpha)=\alpha$. Therefore $\alpha$ is nilpotent. $\blacksquare$
	
	\section{Remarks about the proof}	
	The heart of the proof lies in establishing the Bousfield equivalence of classes $\langle G_j \rangle$ and $\langle G_{j+1} \rangle$. This represents one of the $\omega^2$ tiny steps between $MU$ and $S$.
	
	\vspace{-0.6em}
	
	\begin{center}
		\includegraphics[width=0.9\textwidth]{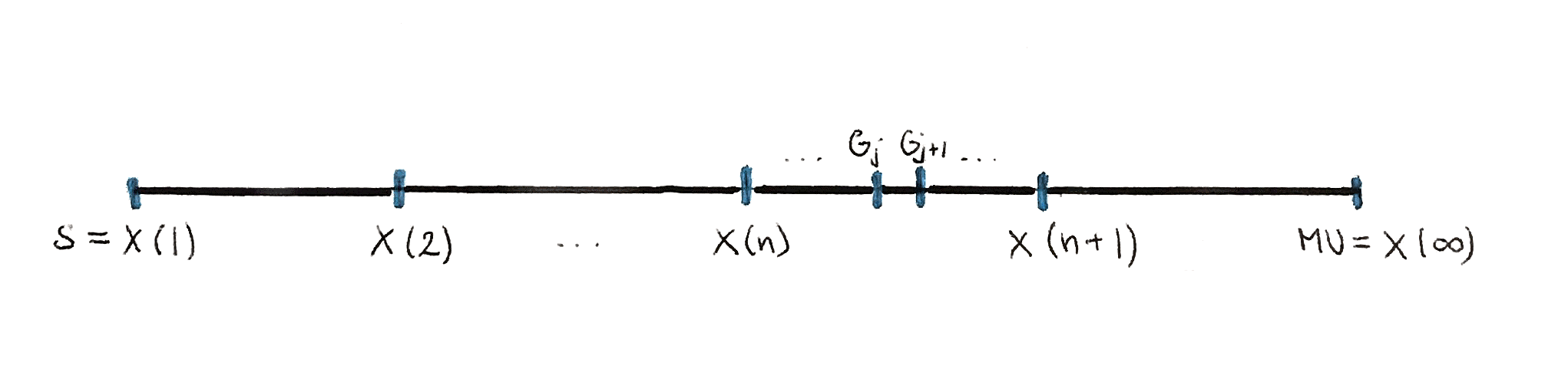}
	\end{center}

	\vspace{-2em}
	
	\noindent One could try to simplify the argument by only making $\omega$ bigger steps of the form $X(n+1) \rightsquigarrow X(n)$ or even a single jump $MU \rightsquigarrow S$. This fails. Descending to the level of $G_j$ is necessary because something remarkable happens in this setting: $\langle G_j \rangle = \langle G_{j+1} \rangle$. The fact that this fails to hold on the level of $X(n)$ or higher suggests that there is little hope of a general argument; any inductive proof of the nilpotence theorem must utilize some specific properties possessed by the $G_j$.
	
	\vspace{1em}
	
	The tool we use to take advantage of these properties is the map $b$ -- it allows us to compare the Bousfield equivalence classes of $G_j$ and $G_{j+1}$. Its definition spans several pages of this essay and may seem convoluted at first, so we try to explain how this map arises naturally. By considering the fibration
	$$ B_{p^j-1} \to B_{p^{j+1}-1} \to J_{p-1}S^{2p^jn+1} $$
	from Lemma \ref{OGfibration} and passing to Thom spectra, one obtains a similar map $G_j \to G_{j+1}$ serving a similar purpose. If one chooses this route, there are other cofibre sequences that need to be established, but the proof proceeds in a similar spirit \cite[see Section 9.5]{ravenel1992nilpotence}. The advantage of the approach taken here and in \cite{devinatz1988nilpotence} is that all of this information is conveniently compressed in the definition of $b$.
	
	\vspace{1em}
	
	The climax of the proof is establishing further properties possessed by $b$. If I tried to distil the insight I have gained by studying the proof in depth, I would say the nilpotence theorem is true because of the existence of the commutative diagram
	$$
	\begin{tikzcd}
	S^{2mp-2} \land G_j \arrow[dr] \arrow[rr, "\gamma \land \id"] & & \Sigma^{\infty} \Omega^2 S^{2m+1}_+ \land G_j \arrow[r, "\mu"]& G_j\\
	& D_p \land G_j \arrow[ur, hook]&& 
	\end{tikzcd}
	$$
	where the long composition is the map $b$. This diagram unites the following information about the map $b$:
	\begin{enumerate}[itemsep=0pt, label=$-$]
		\item Lemma \ref{factorizationOfb}, which provides the factorization in the upper row. This factorization is heavily based on the results from Lemma \ref{OGfibration}.
		\item Lemma \ref{Dpfactorization} following the Snaith's theorem about the splitting structure of $\Sigma^{\infty} \Omega^2  S^{2m+1}_+$, which guarantees that $\gamma$ factors through $D_p$.
	\end{enumerate}
	Once this diagram is established, the contractibility of $b^{-1}G_j$ and hence the proof of the nilpotence theorem follow easily. We also remark that a diagram resembling this one appears in the original proof of the Nishida's theorem and, in fact, the proof presented in this essay was motivated by Nishida's work. We expand on this remark in Chapter 4.
	
	\vspace{1em}
	
	Step II is the only part of the proof of the nilpotence theorem, for which no know alternative proof exists as of 2021.
	
	\subsection{The role of $MU$}
	Intuitively, the nilpotence theorem states that $MU$ detects nilpotence. The spectra $X(n)$ and $F_k$ exhibited in the proof detect nilpotence just as well as the $MU$ does, but this is far from being true for general ring spectra. In this section we inspect the proof and give some sufficient properties for a ring spectrum $T$ to detect nilpotence.
	
	\vspace{1em}
	
	It is clear that if $R$ detects nilpotence and $\langle R \rangle = \langle T \rangle$, then $T$ detects nilpotence too. More generally, we have the following result which has been called the axiomatic nilpotence theorem \cite{mathew}.	
	\begin{theorem}
		Let $R \to T$ be a morphism of ring spectra such that $R$ detects nilpotence. If $T$ is a filtered colimit of spectra $G_j$ such that
		\begin{enumerate}[itemsep=0pt, label=$-$]
			\item the Adams spectral sequence for $G_j \land R$ based on $T$ converges and has vanishing lines of arbitrarily small slopes on the $E_\infty$-page and
			\item $\langle G_j \rangle = \langle R \rangle$ for all $j$,
		\end{enumerate}
		then $T$ detects nilpotence.
	\end{theorem}
	The proof presented in this essay verifies these criteria for the map of spectra $X(n) \to X(n+1)$. Step I corresponds to the first point and Step II corresponds to the second point of this theorem.
	
	\vspace{1em}
	
	The sequel \cite{hopkins1998nilpotence} of the paper \cite{devinatz1988nilpotence} gives a refined characterisation of the spectra detecting nilpotence in terms of Morava's $K$-theories.
	
	\section{Smash product form}
	In this section we deduce the smash product form of the nilpotence theorem from the weak version of the ring spectrum form.
	\begin{theorem}[Nilpotence theorem, smash product form]\label{smashProductForm}
		Let $F$ be a finite spectrum and $f: F \to X$ a map of spectra. If $\id_{MU} \land f$ is null-homotopic, then $f$ is smash nilpotent.
	\end{theorem}
	\begin{proof}
		By Spanier-Whitehead duality we can reduce to the case $F=S$. Indeed, let $DF$ be the Spanier-Whitehead dual of $F$ and let $\widehat{f}: S \to X \land DF$ be the adjoint of $f$. By the properties of Spanier-Whitehead duals we have that
		\begin{center}
				$f$ is smash nilpotent iff $\widehat{f}$ is smash nilpotent
		\end{center}
		and
		\begin{center}
				$\id_{MU} \land f$ is null-homotopic iff $\id_{MU} \land \widehat{f}$ is null-homotopic.
		\end{center}
	
		\noindent It is therefore equivalent to establish the theorem for $\widehat{f}$ and so we can without loss of generality assume that $F=S$.
		
		\vspace{1 em}
		
		Let $f: S \to X$ be a map of spectra such that $\id_{MU} \land f$ is null-homotopic. The diagram
		$$
		\begin{tikzcd}
		S \arrow[r, "f"] \arrow[d, "\eta"]& X \arrow[d, "\eta \land \id_X"] \\
		MU \arrow[r, "\id_{MU} \land f"] & MU \land X
		\end{tikzcd}
		$$
		commutes and any spectrum, in particular $X$, is a homotopy direct limit of its finite subspectra \cite[see A.5.8]{ravenel1992nilpotence}. So $f$ factors as $S \xrightarrow{f} X_\alpha \to X$ through some finite subspectrum $X_\alpha$ of $X$ and similarly the homotopy between $\id_{MU} \land f$ and the constant map factors through $MU \land X_\beta$ for some finite subspectrum $X_\beta$ of $X$. Since the finite subspectra form a direct system we have $X_\alpha = X_\beta$ without loss of generality.
		
		\vspace{1em}
		
		The spectrum $X_\alpha$ is finite so it is $(-d)$-connected for some $d \in \Z$. Then $Y=\Sigma^d X_\alpha$ is $0$-connected and let
		$$ R = \bigvee_{j=0}^\infty Y^{\land j}.$$
		This is a connective ring spectrum of finite type. Under the natural inclusion $Y \hookrightarrow R$ we can consider $S^d \xrightarrow{f} Y \hookrightarrow R$ to be an element of $\pi_d(R)$. Because this element vanishes, it follows that $f$ is smash nilpotent as required.
	\end{proof}

	\section{Ring spectrum form}
	In this short section we deduce the strong version of the ring spectrum form from the smash product form. The adjective \emph{strong} refers to the fact that there are no underlying assumptions on the ring spectrum $R$. In particular, it is not necessarily connective or of finite type while both of these assumptions were required in \ref{weakRingSpectrumForm}. On the other hand, the distinction between the \emph{weak} and the \emph{strong} version is only pedagogical. They are logically equivalent as is witnessed by the circle of implications
	\begin{enumerate}[itemsep=0pt, label=$-$]
		\item Weak ring spectrum form $\implies$ Smash product form (by \ref{smashProductForm})
		\item Smash product form $\implies$ Strong ring spectrum form (by \ref{strongRingSpectrumForm})
		\item Strong ring spectrum form $\implies$ Weak ring spectrum form (trivial)
	\end{enumerate} 
	proven in this essay. With this in mind, we usually only refer to the ring spectrum form and mean the following theorem.
	\begin{theorem}[Nilpotence theorem, ring spectrum form]\label{strongRingSpectrumForm}
		Let $R$ be a ring spectrum and let
		$$h: \pi_*(R) \to MU_*(R)$$
		be the Hurewicz homomorphism. Then every element of $\ker h$ is nilpotent.
	\end{theorem}
	\begin{proof}
		Let $\alpha : S^d \to R$ be an element of $\ker h$. Therefore $\id_{MU} \land \alpha : MU \to MU \land R$ is null-homotopic and hence $\alpha$ is smash nilpotent by the smash product form. So $\alpha$ is nilpotent as required.
	\end{proof}
	For completeness, we pinpoint the places in the proof of the weak version where the additional assumptions of connectivity and finite type were used.
	\begin{enumerate}[itemsep=0pt, label=$-$]
		\item The connectivity of $R$ was used in Proposition \ref{StepI} to establish vanishing lines in the Adams spectral sequence for the $\pi_*(G_j \land R)$.
		\item The finite type of $R$ was used in Lemma \ref{convergence} about the convergence of the Adams spectral sequence.
	\end{enumerate}
	
	\section{Self-map form}
	In this section we deduce the self-map form of the nilpotence theorem from the ring spectrum form.
	\begin{theorem}[Nilpotence theorem, self-map form]\label{selfMapForm}
		Let $X$ be a finite spectrum and let $f: \Sigma^d X \to X$ be a self-map for some $d$. If $MU_*(f)=0$ then $f$ is nilpotent.
	\end{theorem}
	\begin{proof}
		Let $\widehat{f}: S^d \to DX \land X$ be the adjoint of $f$ under the Spanier-Whitehead duality. Since
		\begin{center}
			$f$ is nilpotent iff $\widehat{f}$ is nilpotent
		\end{center}
		it is equivalent to show that $\widehat{f}$ is nilpotent. We establish this by applying the ring spectrum form of the nilpotence theorem with $R=DX \land X$. There are two conditions that need to be checked.
		\begin{enumerate}[itemsep=0pt, label=$-$]
			\item $DX \land X$ is a ring spectrum:
			
			The unit map $\eta: S \to DX \land X$ is the adjoint of $\id_X$ and the multiplication map is given by $m: DX \land X \land DX \land X \xrightarrow{\id_{DX} \land D\eta \land \id_X} DX \land S \land X \simeq DX \land X$ where $D \eta$ is the Spanier-Whitehead dual of $\eta$. It can easily be verified that $\eta$ and $m$ make $DX \land X$ into a ring spectrum by chasing the required diagrams.
			\item $h(\widehat{f}) =0$ where $h$ is the Hurewicz homomorphism for $MU$:
			
			 To see this, note that $MU_*(f)=0$ so $MU \land f^{-1}X$ is contractible. Since $X$ is finite, the composition $\Sigma^{Nd} X \xrightarrow{f^N} X \to MU \land X$ is null-homotopic for some $N$. Then $h(\widehat{f^N}): S^{Nd} \xrightarrow{f} DX \land X \xrightarrow{\eta \land \id} MU \land DX \land X$ is null-homotopic.
		\end{enumerate}
		We have verified the conditions, so we may apply the ring spectrum form of the nilpotence theorem. It now follows that $\widehat{f^N}$ is nilpotent and hence so is its adjoint $f^N$. Therefore $f$ is nilpotent.
	\end{proof}

%
	
	\begin{remark}
		We have shown that the ring spectrum and the smash product forms of the nilpotence theorem are equivalent. The self-map form as stated in the Theorem \ref{selfMapForm} is genuinely a weaker statement.
	\end{remark}

	\chapter{Applications}
	In this chapter we prove Nishida's theorem and address some related questions.
	\section{Nishida's theorem}
	Nishida's theorem is an immediate corollary of the nilpotence theorem for the sphere spectrum $S$. It was proven in 1973 by Goro Nishida \cite{nishida1973nilpotency} and served as motivation for conjecturing and evidence for the truth of the nilpotence theorem.
	\begin{theorem}[Nishida's theorem]
		Every element of positive degree of $\pi_*^S$ is nilpotent.
	\end{theorem}
	\begin{proof}
		Let $d \in \N$. By Serre's finiteness theorem about homotopy groups of spheres all $\pi_d^S$ are finite and hence all elements of $\pi_d^S$ are torsion. On the other hand Novikov \cite{novikov1962homotopy} showed that $\pi_*(MU) \cong \Z\left[x_1, x_2, \ldots \right]$ with $|x_i|=2i$ is torsion-free. Because $h: \pi_*^S \to \pi_*(MU)$ is a ring homomorphism all positive degree elements of $\pi_*^S$ must vanish. By the ring spectrum form of the nilpotence theorem every element of $\ker h$ is nilpotent as required. 
	\end{proof}
	It may be instructive to think about how our proof of the ring spectrum form of the nilpotence theorem specializes to the case $R=S$. It transpires that, despite this being a very special case of the theorem, the proof does not simplify substantially; the complex part of the proof lies in establishing that $\langle G_j \rangle = \langle G_{j+1} \rangle$, a claim which makes no reference to $R$.
	
	\vspace{1em}
	
	Heuristically, one can explain the situation as follows. The proof does not simplify, because it proceeds by constructing a very powerful and robust machinery (the spectra $G_j$ together with a map $b$) designed to solve a specific hard problem. If one does not develop the theory in its entirety, it just does not work. A colourful analogy comes to mind -- even an otherwise impeccable Ferrari without a single important part, for example a clutch pedal, is unable to move. However, if one is only interested in driving around Cambridge, a bicycle may be a more suitable mode of transportation than a fully functioning Ferrari anyway.
	
	This brings us to the obvious question. Can we adapt the existing proof to produce a simpler one that cannot necessarily prove the nilpotence theorem, but is still powerful enough to do the case $R=S$? Yes, this can be done and it is very much in the spirit of Nishida's original proof. Instead of the contractibility of the telescope $b^{-1}G_j$ we can show the contractibility of $\alpha^{-1}S$ which immediately implies that $\alpha$ is nilpotent.
	
	 \vspace{1 em}
	 
	In the proof of the nilpotence theorem we use the factorization of the map
	$$b^N : \Sigma ^{N|b|} G_j \to D_{Np} \land G_j \to G_j$$
	through the spaces $D_{Np}$ with the homotopy colimit $\varinjlim D_{Np} = H \F_p$ to establish that the identity map on the telescopes passes through $H \F_p \land b^{-1}G_j \simeq *$. This is completely analogous to what Nishida did in his proof. For certain elements $\alpha \in \pi_d^S$ he utilized factorizations of the form
	$$\alpha^N: S^{Nd} \to D_{Np} \to S$$
	previously established by Toda \cite{toda1968extended} to show that $\alpha$ is nilpotent. The nilpotence theorem is thus a vast generalization of both the statement and the proof of Nishida's theorem.
	\begin{remark}
		To be completely historically accurate, we have to point out that Nishida's paper \cite{nishida1973nilpotency} contains two different approaches towards the proof. One of the approaches establishes Nishida's theorem and the other only produces some partial results (for example, the elements of $\pi_*^S$ of prime order are nilpotent). The paper \cite{devinatz1988nilpotence} generalizes the latter approach. 
	\end{remark}
	
%
%
%
%
%
%
%
%
%
	
	\bibliography{mybib}
	\bibliographystyle{ieeetr}
\end{document}